\documentclass[12pt,draftcls,onecolumn]{IEEEtran}

\usepackage{hyperref}  % make clickable links in LaTeX :-)
\hypersetup{hypertexnames=false}

\usepackage{xcolor}
\usepackage{pdfsync}
\usepackage{graphicx}
\usepackage{epsfig}
\usepackage{amssymb}
\usepackage{amsthm}
\usepackage{mathtools}
\usepackage{amsfonts}
\usepackage{lineno}
\usepackage{tikz}
\usetikzlibrary{positioning,shapes}
\usetikzlibrary{arrows}
\usepackage{dsfont}
\usepackage{algorithm}
\usepackage{algorithmicx}
\usepackage{algpseudocode}
\usepackage{amsmath}
\allowdisplaybreaks
\newtheorem{theorem}{Theorem}
\newtheorem{definition}{Definition}
\newtheorem{corollary}{Corollary}
\newtheorem{lemma}{Lemma}
\newtheorem{assumption}{Assumption}
\newtheorem{proposition}{Proposition}
\newtheorem{remark}{Remark}

\newtheorem{problem}{Problem}

\newcommand{\norm}[1]{\left\Vert #1\right\Vert}

\newcommand{\abs}[1]{\left|#1\right|}

\def\field#1{\mathbb #1}%
\def\R{\field{R}}%
\def\N{\field{N}}%
\def\Z{\field{Z}}%
\def\B{\field{B}}%
\newcommand{\X}{\ensuremath{\mathcal X}}

\newcommand{\U}{\ensuremath{\mathcal U}}

\newcommand{\Rn}[1][n]{\R^{#1}}
\newcommand{\Rp}{\R_{\geq 0}}
\newcommand{\Rsp}{\R_{> 0}}
\newcommand{\Zp}{\Z_{\geq 0}}
\def\K{\mathcal{K}}%
\DeclareMathOperator{\id}{id}

\def\KL{\mathcal{KL}}%
\def\Kinf{\mathcal{K}_\infty}%
\let\ol=\overline%
\let\ul=\underline%

\def\Ni{\mathcal{N}_i}

\usepackage{color}

\newcommand{\nom}{\mathrm{nom}}

\usepackage{xspace}	
\providecommand{\fsCLF}{fs-CLF\xspace}

\providecommand{\fsCLFs}{fs-CLFs\xspace}

\usepackage{enumerate}

\title{Distributed model predictive control via finite-step control Lyapunov functions}

\author{
  Navid~Noroozi\thanks{Navid~Noroozi is with SIGNON Deutschland GmbH (DB InfraGO AG), Europaplatz 2, 10557 Berlin, Germany, \texttt{navid.n.noroozi@deutschebahn.com}}, 
    Maryam Sharifi\thanks{Maryam Sharifi is with ABB Corporate Research, Västerås, Sweden, \texttt{maryam.sharifi@se.abb.com}}
}

\begin{document}
\maketitle
%\tableofcontents

\begin{abstract}
As opposed to classical converse Lyapunov theorems, finite-step converse results are constructive and offer a different starting point:
for sufficiently large finite step ahead, say $M$, in an explicit sense, 
any scaled norm can serve as a converse finite-step Lyapunov function. 
As for interconnected discrete-time systems, 
similar lines of argument lead to ``non-conservative'' small-gain conditions.
Motivated by this viewpoint, 
this paper develops a distributed model predictive control framework for constrained interconnected nonlinear discrete-time systems.
Each subsystem solves one local optimization problem at each system time step instant by setting the local stage function in form of a local control finite-step like Lyapunov function, 
using time-aligned neighbor predictions, optimized state-constraint tightening radii, and a finite-step small-gain terminal inequality.
Since received neighbor predictions need not equal the trajectories generated by future receding-horizon optimizations, the nominal converse certificates do not alone ensure recursive feasibility or stability.
We therefore develop shift-compatible constraint margins, a local one-step terminal feasibility test for networks that are affine in control, and an analytical bound for the prediction and reoptimization mismatch.
The resulting analysis gives recursive feasibility, constraint satisfaction, and a practical $M$-step Lyapunov estimate, with asymptotic convergence when the prediction and reoptimization mismatch bound tends to zero.
For constrained linear networks, the conditions reduce to finite-dimensional matrix, QP, and SOCP tests.
The framework is specialized to current sharing and terminal bus voltage safety under DC/DC power converters' operational constraints in a two-DGU DC microgrid evaluated on a small laboratory-scale prototype.
\end{abstract}


\section{Introduction}

Large-scale engineered systems are increasingly organized as networks of dynamically coupled subsystems.
Their control design must reconcile local computational scalability and limited communication with explicit state and input constraints and a network-level stability guarantee.
Distributed model predictive control (DMPC) is a natural framework for this purpose because it decomposes online optimization while retaining model-based prediction and constraint handling \cite{Scattolini.2009,Christofides.2013}.

The central motivation of this work is to enjoy the constructive converse viewpoint offered by finite-step Lyapunov theory.
While classical converse Lyapunov theorems are mainly existential, for discrete-time systems, the finite-step converse Lyapunov theorems, e.g., \cite[Theorem~13]{Geiselhart.2014c} is constructive in a more useful sense: 
if an asymptotic stability decay rate $\beta$ satisfies $\beta(r,M)<r$ for a sufficiently large integer $M$, then every function of the form $\eta(\norm{x})$, $\eta\in\Kinf$, is a finite-step Lyapunov function.
In particular, for a globally exponentially stable system with $\norm{x(t,\xi)}\leq C\mu^t\norm{\xi}$, any $M$ satisfying $C\mu^M<1$ is admissible.

The same idea extends to interconnected discrete-time systems.
Non-conservative finite-step small-gain theorems allow the local functions to decrease only after $M$ steps and do not require every isolated subsystem to be one-step stable.
Under the certain decay rate assumptions, e.g. cf.~\cite[Theorems~IV.5 and IV.9]{Geiselhart.2015}.
In particular for globally exponentially stable networks, one may choose simple local norm functions, increase $M$ until the finite-step gains satisfy the cyclic small-gain condition, and explicitly construct a Lyapunov function for the overall interconnection~\cite{Geiselhart.2015,Gielen.2015}.
The term \emph{non-conservative} is used here in this precise converse sense: on the covered system classes, the finite-step small-gain conditions are sufficient and necessary.
It does \emph{not} imply that every numerical bound introduced later for our DMPC setting---such as Lipschitz estimates, uncertainty radii, or admissibility margins---is itself non-conservative.

This constructive setting motivated the centralized fs-CLF-based multi-step MPC of \cite{Noroozi.2020}.
There, a finite-step control Lyapunov function is used as the stage cost, a hard $M$-step contraction is imposed, and the associated admissible finite-step feedback provides a feasible candidate for the optimization problem.
The present paper extends this synthesis principle from a single system to a network of constrained systems.
The extension is nontrivial because a local MPC controller does not know the neighboring trajectories that will actually be generated by future receding-horizon optimizations; it only receives announced predictions.
Consequently, the nominal finite-step Lyapunov and small-gain certificates do not by themselves establish recursive feasibility or stability of the distributed implementation.  The prediction mismatch must be represented explicitly, and feasibility at the next sampling instant must be proved despite packet variation and reoptimization.

The architecture considered here is parallel, non-iterative, and non-cooperative.
At each sampling instant, every subsystem performs one local optimization using one set of received neighbor packets; no inner negotiation is used to recover a network-wide optimum, and each stage cost depends only on the local state.
Non-iterative and non-cooperative DMPC are established directions \cite{Dunbar.2007,FarinaScattolini.2012}.
The closest Lyapunov-based predecessor is the almost-decentralized nonlinear MPC of Hermans et al. \cite{Hermans.2010}, which also treats dynamically coupled discrete-time nonlinear subsystems with one neighbor-information exchange per sample.  Its local controller enforces a one-step max-type control-Lyapunov inequality, supplemented by a Lyapunov--Razumikhin window.  Our distinction is therefore not simply a longer prediction horizon: the present design is rooted in constructive $M$-step converse theory and in cyclic finite-step small-gain conditions that may remain applicable even when one-step subsystem Lyapunov inequalities are unavailable.

Lyapunov and small-gain arguments also appear in other decentralized and distributed MPC designs \cite{MagniScattolini.2006,RaimondoISS.2007,RaimondoIterative.2009}.
The recent small-gain-based DMPC of Zheng et al. \cite{ZhengSmallGain.2025} is particularly related in spirit because it permits non-simultaneous local Lyapunov decrease.  Their setting, however, is continuous-time and is built around pre-existing local auxiliary controllers and instantaneous ISS-Lyapunov derivative inequalities.  The resulting small-gain theorem is used as a sufficient stability condition; it does not provide the discrete-time constructive converse mechanism by which simple local norm candidates and admissible gains are recovered by increasing $M$.  Nor does it address the prediction-mismatch and recursive-feasibility problem created by one-shot exchange of receding-horizon neighbor trajectories.

A second closely related branch uses trajectory contracts or consistency sets.
Lucia et al.~\cite{Lucia.2015} exchange set-valued contracts that contain future coupling trajectories, propagate reachable sets, impose nested contract consistency, and use robust terminal invariant sets and terminal costs.  That framework is more general with respect to set-valued coupling uncertainty and shared constraints.  
In our work, an exchanged packet is deliberately lighter: it consists of a nominal trajectory and scalar radii.
Recursive feasibility is obtained without a terminal set, through a time-aligned shifted candidate, computable state-constraint margins, and a local one-step terminal feasibility problem.
The terminal condition is the finite-step small-gain inequality motivated by the constructive converse finite-step control Lyapunov theorem.

The specific contributions are as follows.
\begin{enumerate}
\item We formulate a non-iterative DMPC architecture that transfers constructive finite-step Lyapunov and non-conservative small-gain ingredients to local receding-horizon problems.  The novelty is not the formal use of $W_i$ as a stage function or the placement of an $M$-step inequality at the horizon endpoint; it is the treatment required to preserve a meaningful finite-step certificate when each controller has access only to uncertain neighbor predictions.
\item We separate the optimized state-constraint tightening radius, the packet-to-packet shift error, and the actual-to-nominal error accumulated under repeated reoptimization.
For nonlinear systems which are affine in control with stat space defined by polyhedral, and control space defined by polytopes, shift-compatible tightened constraints and a one-step terminal feasibility test yield recursive feasibility and closed-loop constraint satisfaction without imposing a terminal set.
An additional bound on the prediction and reoptimization mismatch then produces a practical $M$-step Lyapunov estimate and asymptotic convergence when the terminal discrepancy tends to zero.
\item For constrained linear networks, we reduce the abstract hypotheses to finite-dimensional tests.
Finite-step gains follow from an $M$-step closed-loop map or semidefinite certificates; the cyclic small-gain condition becomes a linear program in logarithmic scaling variables; and admissibility, shift-margin, terminal-feasibility, and regularity conditions reduce to explicit matrix inequalities, QPs, SOCPs, or scalar interval tests.
\item We specialize the framework to current sharing and terminal bus voltage safety with actuator constraints in a two-DGU DC microgrid.
The closed-set finite-step certificate is constructed for the current-sharing objective, while bus voltage and converter commands are enforced as hard constraints.
Synchronized simulation and hardware experiments assess current sharing, load estimation, constraint satisfaction, solver health, communication timing, and real-time implementability.
\end{enumerate}

The DC-microgrid application complements established consensus, droop-modification, passivity, and secondary averaging methods \cite{Nasirian.2014,TripExperimental.2018} by retaining the electrical coupling in the prediction model and enforcing converter and bus-voltage limits directly.  It also extends our earlier centralized finite-step MPC study \cite{Noroozi2018IFACDCMicrogridMPC} to a distributed implementation with explicit recursive-feasibility and reoptimization analysis.
A customizable Python-native simulation environment is appended to this paper as proof of concept:
Check the git repository \href{https://github.com/navidnoroozi/dmpc-4-dcmg}{\texttt{dmpc-4-dcmg}}, follow the instructions \href{https://github.com/navidnoroozi/dmpc-4-dcmg/blob/main/Python_env/HOW_TO_RUN_WORKFLOW.md}{\texttt{README}} supported with a simulation workflow (copy/paste) guideline \href{https://github.com/navidnoroozi/dmpc-4-dcmg/blob/main/Python_env/HOW_TO_RUN_WORKFLOW.md}{\texttt{HOW-TO-RUN-WORKFLOW}}.

The remainder of the paper is organized as follows.  Section~\ref{sec:preliminaries} introduces notation and reviews the finite-step control-Lyapunov and small-gain ingredients needed later.
Section~\ref{sec:Distributed MPC for interconnected systems} presents the distributed optimal-control problem and communication architecture.
Section~\ref{sec:dist-fsclf-stability} establishes recursive feasibility and finite-step stability.
Section~\ref{sec:Constrained linear networks} gives constructive certificates for constrained linear networks.  Section~\ref{sec:case_study_twodgu} specializes the framework to the two-DGU DC microgrid and reports simulation and hardware results.  The final section concludes the paper.

\section{Preliminaries} \label{sec:preliminaries}
This section introduces the notation and the controlled system.
It also reviews the required stability and admissibility concepts.
the admissible finite-step feedback, finite-step converse Lyapunov functions and small-gain notions used later.
A more detailed development of admissible finite-step feedbacks and fs-CLFs can be found in~\cite[Section~3]{Noroozi.2020}.

\subsection{Notation} \label{sec:notation}

In this paper, $\Rp (\Rsp)$ and $\Zp (\N)$ denote the nonnegative
(positive) real numbers and the nonnegative (positive) integers,
respectively.
For a set $\mathcal{S} \subseteq \Rn$, $\operatorname{int}(\mathcal{S})$ and $\operatorname{co}(\mathcal{S})$, respectively, denote the interior and the convex hull of $\mathcal{S}$.
Given $\mathcal{S} \subseteq \Rn$, $\mathcal{S}^\ell := \underbrace{\mathcal{S} \times \dots \times \mathcal{S}}_{\ell \, \, \mathrm{times}}$.
The $i$th component or partition of $v \in \Rn$ is denoted by $v_i$.
%For any $v, w \in \Rn$, we write $v \gg w$ ($v \geq w$) if and only if $v_i > w_i$ ($v_i \geq w_i$) for each $i \in \{1,\dots,n\}$. If $v\geq w$ but $v\ne w$ we write $v > w$. 
%A vector $v\in\R^n$ is \emph{positive} if $v\gg 0$.
%The relation $v \ngeq w$ holds if and only if there exists some $i \in \{1,\dots,n\}$ such that $v_i < w_i$.
For any $v \in \Rn$, $v^\top$ denotes its transpose.
We write $(v,w)$ to represent $[v^\top,y^\top]^\top$ for $x \in \Rn,y \in \R^p$.
In a more general way, the vector operator $\operatorname{col}_{i=1}^{\ell}(v_i)$ vertically stacks the vectors $v_1\in\R^{n_1},\dots,v_\ell\in\R^{n_1}$ into one long column vector.
For $x \in \Rn$, we, respectively, denote the Euclidean norm and the maximum norm by $\norm{x}$ and by $\abs{x}_\infty$.
For $v_{\Ni}:=\operatorname{col}_{j\in\Ni}(v_j)$, we define $\|v_{\Ni}\|_{\Ni,\infty}:=\max_{j\in\Ni}\|v_j\|$.
The identity function is denoted by $\id$.
Composition of functions is denoted by the symbol $\circ$ and repeated composition of, e.g., a function $\gamma$ by $\gamma^{i}$.
For positive definite functions $\alpha,\gamma$ we write $\alpha<\gamma$ if $\alpha(s)<\gamma(s)$ for all $s>0$.

A set \(\mathcal{S}\subseteq\Rn\) is called a nonempty polyhedron if there exist \(m\in\N\),
\(F\in\mathbb{R}^{m\times n}\), and \(f\in\mathbb{R}^m\) such that
$\mathcal{S}
=
\left\{
z\in\Rn
\,\middle|\,
F_i z\leq f_i
\quad\forall i\in\{1,\ldots,m\}
\right\} \neq \varnothing$,
where \(F_i\) denotes the \(i\)th row of \(F\) and \(f_i\in\R\) denotes the \(i\)th component of \(f\).
Without loss of generality, we assume that $F_i \neq 0$.
We note that a polyhedron $\mathcal{S}$ is automatically closed by definition.
If a polyhedron $\mathcal{S}$ is bounded, it is called a polytope.
For a set $\mathcal{S}\subseteq\Rn$ and a set $\mathcal{E}\subseteq\Rn$ define the
\emph{Pontryagin difference} (set subtraction)
\[
\mathcal{S} \ominus \mathcal{E} := \{z\in\Rn \;:\; z+e\in \mathcal{S} \ \ \forall e\in \mathcal{E}\}.
\]
Let $\mathbb{B}_\Delta:=\{e\in\Rn:\norm{e}\leq \Delta\}$ be the closed ball of radius $\Delta$.
A positive definite (semidefinite) matrix $P \in \R^{n\times n}$ is denoted by $P \succ 0$ ($P \succeq 0$). 

\subsection{System description}
Consider an interconnection of $\ell$ subsystems
\begin{equation}\label{eq:network_decoupled}
 \Sigma_i:\quad x_i(t+1)=g_i\bigl(x_i(t),x_{\Ni}(t),u_i(t)\bigr),
\end{equation}
where $x_i\in\X_i\subseteq\R^{n_i}$, $u_i\in\U_i\subseteq\R^{m_i}$, and $\Ni\subset\{1,\ldots,\ell\}\setminus\{i\}$ is the set of physical neighbors affecting subsystem $i$.
We use $x_{\Ni}:=\operatorname{col}_{j\in\Ni}(x_j)$, $\X_{\Ni}:=\prod_{j\in\Ni}\X_j$.
With $\X:=\prod_{i=1}^{\ell}\X_i$, $\U:=\prod_{i=1}^{\ell}\U_i$, $x:=\operatorname{col}_{i=1}^{\ell}(x_i)$, $u:=\operatorname{col}_{i=1}^{\ell}(u_i)$, and $g:=\operatorname{col}_{i=1}^{\ell}(g_i)$, the network is
\begin{equation}\label{eq:network}
 \Sigma:\quad x(t+1)=g\bigl(x(t),u(t)\bigr),
\end{equation}
where $g(0,0)=0$.  The state constraint $x\in\X$ is imposed on admissible trajectories; hence we regard $g:\X\times\U\to\R^n$ rather than assuming a priori that every input maps $\X$ into itself.  We assume that $g$ is $\K$-bounded on $(\X,\U)$, as defined below.

\begin{definition}\label{def:gKb} \cite{Noroozi.2020}
The map $g$ is \emph{$\K$-bounded} on $(\X,\U)$ if there exist $\kappa_1,\kappa_2\in\K$ such that
\[
 \norm{g(\xi,\mu)}\leq\kappa_1(\norm{\xi})+\kappa_2(\norm{\mu})
 \qquad \forall (\xi,\mu)\in\X\times\U.
\]
\end{definition}
For input-to-state systems, this assumption is necessary; see \cite[Lemma~5]{Geiselhart.2017}.
Componentwise, we assume that for each $i$ there exist $\kappa_{1,i},\kappa_{2,i}\in\K$ such that
$ \norm{g_i(\xi_i,\xi_{\Ni},\mu_i)}
 \leq \kappa_{1,i}\bigl(\norm{(\xi_i,\xi_{\Ni})}\bigr)
      +\kappa_{2,i}(\norm{\mu_i})$
for all $(\xi_i,\xi_{\Ni},\mu_i)\in\X_i\times\X_{\Ni}\times\U_i$.

For the DMPC analysis we assume, componentwise, that each local map satisfies the incremental Lipschitz estimate
\begin{align}\label{eq:Lipschitz_dynamics_p}
\|g_i(\xi_i,\xi_{\Ni},\mu_i)-g_i(\xi_i',\xi_{\Ni}',\mu_i')\| \leq & L_{i,1}\|\xi_i-\xi_i'\| \nonumber\\
 & +L_{i,2}\|\xi_{\Ni}-\xi_{\Ni}'\|_{\Ni,\infty}
 +L_{i,3}\|\mu_i-\mu_i'\|,
\end{align}
with $L_{i,1},L_{i,2}>0$ and $L_{i,3}\geq0$ on the stated constraint domain.
This estimate is used for prediction-error propagation in Sections~\ref{sec:Distributed MPC for interconnected systems} and~\ref{sec:dist-fsclf-stability}; it is not part of the converse finite-step Lyapunov or small-gain results themselves.  No additional continuity assumption on $g_i$ is required below, because \eqref{eq:Lipschitz_dynamics_p} already implies continuity of $g_i$ on this domain.

\subsection{Admissible finite-step feedback control}
For a finite input sequence $\mathbf u_M=(u_0,\ldots,u_{M-1})\in\U^M$, let
\begin{equation}\label{eq:finite_horizon_solution}
 x(0,\xi,\mathbf u_M):=\xi,
 \qquad
 x(l+1,\xi,\mathbf u_M):=g\bigl(x(l,\xi,\mathbf u_M),u_l\bigr)
\end{equation}
for $l=0,\ldots,M-1$.
Let $q:\X\to\U^M$ be a block feedback with $q(\xi)=(q^{[0]}(\xi),\ldots,q^{[M-1]}(\xi))$.
When $q$ is evaluated every $M$ steps, the induced input and state sequences satisfy
\begin{equation}\label{eq:block_feedback_trajectory}
 \begin{aligned}
 u_q(kM+l)&=q^{[l]}(x_q(kM)),\\
 x_q(kM+l+1)&=g\bigl(x_q(kM+l),u_q(kM+l)\bigr),
 \end{aligned}
\end{equation}
for $k\in\Zp$ and $l=0,\ldots,M-1$, with $x_q(0)=\xi$.
Equivalently, $x(\cdot,\xi,u_q)=x_q(\cdot)$.

\begin{definition}\label{def:admissible_finite_step_feedback}
A map $q:\X\to\U^M$ is an \emph{admissible finite-step feedback control law} of length $M$ if, for every $\xi\in\X$ and $l=1,\ldots,M$,
\begin{enumerate}
\item $x(l,\xi,q(\xi))\in\X$;
\item there exists $\kappa_l\in\K$ such that
$\norm{x(l,\xi,q(\xi))}\leq\kappa_l(\norm{\xi})$ for all $\xi\in\X$.
\end{enumerate}
\end{definition}
Thus admissibility combines constraint preservation over the complete block with bounded intermediate growth.  The expanded recursion used in \cite[Section~3]{Noroozi.2020} is equivalent to \eqref{eq:finite_horizon_solution}--\eqref{eq:block_feedback_trajectory}.

\subsection{Stabilizability of network $\Sigma$}\label{sec:finite-step-control Lyapunov-function}
\begin{definition}\label{def:ISS-stabilization}
An admissible finite-step feedback $q$ asymptotically stabilizes \eqref{eq:network} in $\X$ if there exists $\beta\in\mathcal{KL}$ such that
$\norm{x(t,\xi,u_q)}\leq\beta(\norm{\xi},t)$ for all $\xi\in\X$, $t\in\Zp$.
If $\beta(r,t)=C\sigma^t r$ for some $C\geq1$ and $\sigma\in[0,1)$, the stabilization is exponential.
\end{definition}

\begin{definition}\label{def:fsCLF}
Let $M\in\N$ and let $V:\R^n\to\Rp$ be continuous on $\Rn$ and there exist $\ul\alpha,\ol\alpha\in\Kinf$ such that
$\ul\alpha(\norm{\xi})\leq V(\xi)\leq\ol\alpha(\norm{\xi})$
for all $\xi\in\Rn$.
The function $V$ is a \emph{finite-step control Lyapunov function} (fs-CLF) for \eqref{eq:network} if there exist an admissible finite-step feedback control laws $q:\X\to\U^M$ and $\alpha\in\Kinf$, $\alpha<\id$, such that
\begin{equation}\label{eq:decayV}
 V\bigl(x(M,\xi,u_q)\bigr)\leq\alpha(V(\xi))
 \qquad \forall \xi\in\X.
\end{equation}
\end{definition}
For $M=1$, Definition~\ref{def:fsCLF} reduces to a classic control Lyapunov function.  For $M>1$, intermediate increases are permitted; admissibility controls those intermediate states, while \eqref{eq:decayV} supplies contraction at block times.

\begin{proposition}\label{prop:P}\cite[Proposition~13]{Noroozi.2020}
Let $V$ be an fs-CLF and let $q$ be its associated admissible finite-step feedback control laws.
Then $q$ asymptotically stabilizes \eqref{eq:network} in $\X$.
\end{proposition}

The constructive converse result that motivates the present control design can be stated compactly as follows.  Consider an autonomous, $\K$-bounded system
\begin{equation}\label{eq:autonomous-system}
 x(t+1)=\hat g(x(t))
\end{equation}
whose solutions satisfy
$ \norm{x(t,\xi)}\leq\beta(\norm{\xi},t)$,
$ \beta\in\mathcal{KL}$.

\begin{proposition}\label{thm:converse-Lyap} \cite[Theorem~13]{Geiselhart.2014c}
Suppose there exists $M\in\N$ such that
\begin{equation}\label{eq:beta-condition}
 \beta(r,M)<r \qquad \forall r>0.
\end{equation}
Then, for every $\eta\in\Kinf$,
\begin{equation}\label{eq:converse-Lyap}
 V(\xi):=\eta(\norm{\xi})
\end{equation}
is a finite-step Lyapunov function for \eqref{eq:autonomous-system}, with $\ul\alpha=\ol\alpha=\eta$ in the comparison bounds and decay rate $\alpha(\cdot)=\eta\circ\beta(\eta^{-1}(\cdot),M)<\id$. If
$\beta(r,t)=C\sigma^t r$, then any $M$ satisfying $C\sigma^M<1$ fulfills \eqref{eq:beta-condition}.
\end{proposition}

An immediate conclusion of Proposition~\ref{thm:converse-Lyap} is that under the stabilizability decay rate \eqref{eq:beta-condition}, any (linearly) scaled norm can be chosen as \fsCLF candidate for which an admissible finite-step feedback control $q$ exists.
In \cite[Problem~10, Algorithm~11, and Proposition~13]{Noroozi.2020}, this fs-CLF is used as the stage cost and in a hard $M$-step contraction constraint to find such a finite-step feedback control $q$.
In that way, the feasibility of the associated optimization problem is certified by design.
This is the constructive link between the converse Lyapunov function result and multi-step MPC.
We aim to extend this idea for development of a distributed MPC scheme to control network~\eqref{eq:network}.
The next piece of tools that we need is the small-gain conditions (SCGs) as discussed below.

\subsection{Non-conservative small-gain condition}
The network counterpart uses local finite-step Lyapunov-like functions and cyclic SGCs.

\begin{assumption}\label{ass:clf-small-gain}
Let $M\geq1$.
Suppose there exist continuous functions $W_i:\R^{n_i}\to\Rp$, an admissible finite-step block feedback
$q=(q_1,\ldots,q_\ell):\X\to\U^M$ with components $q_i:\X\to\U_i^M$, and gains
$\gamma_{ij}\in\Kinf\cup\{0\}$ such that:
\begin{enumerate}[(i)]
\item for each $i$, there exist $\ul\alpha_i,\ol\alpha_i\in\Kinf$ satisfying
\begin{equation}\label{eq:Wi-bounds}
 \ul\alpha_i(\norm{\xi_i})\leq W_i(\xi_i)
 \leq\ol\alpha_i(\norm{\xi_i});
\end{equation}
\item with $u_q$ denoting the block input induced by $q$,
\begin{equation}\label{eq:Wi-estimate}
 W_i\bigl(x_i(M,\xi,u_q)\bigr)
 \leq\max_{j\in\{1,\ldots,\ell\}}\gamma_{ij}(W_j(\xi_j))
 \qquad \forall \xi\in\X;
\end{equation}
\item the gains in \eqref{eq:Wi-estimate} satisfy
\begin{equation}\label{eq:SGC}
 \gamma_{i_1i_2}\circ\gamma_{i_2i_3}\circ\cdots\circ
 \gamma_{i_ri_1}<\id
\end{equation}
for every cycle $(i_1,\ldots,i_r)$ with distinct indices and $r=1,\ldots,\ell$.
\end{enumerate}
\end{assumption}

\begin{proposition}\label{prop:SG-NGGRW}
If Assumption~\ref{ass:clf-small-gain} holds, then the finite-step feedback control laws $q$ asymptotically stabilizes \eqref{eq:network} in $\X$.
\end{proposition}
\begin{proof}
The cyclic condition yields small-gain scaling functions from which a global finite-step Lyapunov function is constructed~\cite[Theorem~IV.1]{Geiselhart.2015}.
Proposition~\ref{prop:P} then gives asymptotic stability.
\end{proof}

The \emph{converse} result is the essential motivation for using Assumption~\ref{ass:clf-small-gain} in the present DMPC design.
Under the hypothesis of \cite[Theorem~IV.5]{Geiselhart.2015}, the local functions $W_i$ may be chosen as norms and, for every sufficiently large $M$ satisfying the stated decay bound, gains can be constructed that satisfy both \eqref{eq:Wi-estimate} and the cyclic small-gain condition.
Hence the finite-step functions and small-gain conditions are justified by a \emph{constructive} converse argument rather than postulated solely as terminal ingredients.
Additional radii, the Lipschitz bounds, shift margins, and one-step terminal feasibility conditions introduced below are needed because distributed MPC has access to announced neighbor predictions rather than the future trajectories to which the nominal converse stability/stabilizability theorem applies.

\section{Distributed MPC via fs-CLF}\label{sec:Distributed MPC for interconnected systems}
In our DMPC implementation, every subsystem
$\Sigma_i$ receives a predicted trajectory from each physical neighbor and solves one local
finite-horizon problem at every sampling instant.  No precomputed ancillary feedback law is
used.  Instead, nonnegative radius variables are optimized online and are used to tighten the
nominal state constraints.  A terminal finite-step inequality, constructed from the local
\fsCLFs and the small-gain functions, supplies the contractive ingredient.

An important distinction is required throughout this section.  The optimized radii determine the state-constraint tightening used by the OCP.  Under repeated reoptimization they do
not, by themselves, bound the difference between the actual closed-loop trajectory and the
complete open-loop prediction computed at an earlier time.  The additional error caused by
future neighbor-packet errors and by replacing the planned inputs with newly optimized first
moves is quantified explicitly in Section~\ref{subsec:online_tube_stability}.

\subsection{Distributed communication requirements}

At each $t\in\Zp$, subsystem $\Sigma_i$ receives from every $j\in\Ni$ a predicted trajectory
\[
\bar{\bf x}_j(t)=\bigl(\bar x_j(0|t),\ldots,\bar x_j(M-1|t)\bigr),
\qquad \bar x_j(0|t)=x_j(t).
\]
The collection over $\Ni$ is denoted by
$\bar{\bf x}_{\Ni}(t):=(\bar{\bf x}_j(t))_{j\in\Ni}$.
Given a local input sequence
${\bf u}_{M,i}(t):=(u_i(0|t),\ldots,u_i(M-1|t))$, the nominal dynamics are
\begin{align}
\label{eq:nom-sys}
 x_i^{\nom}(0|t)&=x_i(t),\nonumber\\
 x_i^{\nom}(l+1|t)&=g_i\bigl(x_i^{\nom}(l|t),\bar x_{\Ni}(l|t),u_i(l|t)\bigr),
 \quad l=0,\ldots,M-1,
\end{align}
while the actual subsystem evolves according to~\eqref{eq:network_decoupled}.

To mitigate the mismatch between the actual state trajectories of the neighbors and their received predictions, at each time step $t$ each local OCP is solved in a tightened local state-space by radius $r_i$.
We aim to tighten the local state-space $\X_i$ minimally.
Therefore, the OCP decision vector also contains
\[
 {\bf r}_{M-1,i}(t):=\bigl(r_i(1|t),\ldots,r_i(M-1|t)\bigr),
 \qquad r_i(l|t)\geq0.
\]
For each $t$, we set $r_i(0|t):=0$ by design; this fixed value is not an optimization variable.
For $l=0,\ldots,M-1$, define the nominal (tightened) local state-space $\X^\nom$ at each optimization stage by
\begin{equation}\label{eq:stage_tight_X}
 \X_i^{\nom}(l|t)
 :=\{z\in\R^{n_i}:z+\B_{r_i(l|t)}\subseteq\X_i\}
 =\X_i\ominus\B_{r_i(l|t)}.
\end{equation}
We \emph{additionally} compute $r_i(M|t)$ which is \emph{not} an OCP decision variable.
It is the uniquely derived
auxiliary terminal radius defined after the OCP by the final radius recursion, cf.
\eqref{eq:terminal_radius_aux} below.

At time $t$, controller $i$ forms its packet by a one-step shift-and-insert operation:
\begin{equation}\label{eq:packet_shift_x}
 \bar x_i(0|t):=x_i(t),\qquad
 \bar x_i(l|t):=x_i^{\nom,\star}(l+1|t-1),
\end{equation}
\begin{equation}\label{eq:packet_shift_r}
 \bar r_i(0|t):=0,\qquad
 \bar r_i(l|t):=r_i^\star(l+1|t-1),
\end{equation}
for $l=1,\ldots,M-1$.
Thus the entry with prediction index $l$ refers to the same absolute
time $t+l$ as the nominal state used by the receiving controller.  In particular,
$\bar r_i(M-1|t)$ uses the auxiliary terminal radius computed at time $t-1$.

\begin{assumption}[Prediction exchange]\label{ass:packets}
Fix $M\in\N$.  At every time $t$, controller $i$ broadcasts
\begin{equation}\label{eq:one_step_lag_packet}
 \mathcal P_i(t):=\bigl(\bar{\bf x}_i(t),\bar{\bf r}_i(t)\bigr)
\end{equation}
to its physical neighbors.  Controller $i$ receives
$\{\mathcal P_j(t)\}_{j\in\Ni}$ before solving its local OCP.  The packet entries are generated
by \eqref{eq:packet_shift_x}--\eqref{eq:packet_shift_r}; initial packets at $t=0$ are supplied by
a feasible initialization procedure.
\end{assumption}

Motivated by Assumption~\ref{ass:clf-small-gain}, subsystem $i$ imposes
\begin{equation}\label{eq:terminal_ineq_global_max}
 W_i\bigl(x_i^{\nom}(M|t)\bigr)
 \leq\max_{j\in\{1,\ldots,\ell\}}\gamma_{ij}\bigl(W_j(x_j(t))\bigr).
\end{equation}
The required scalar communication is determined by the sparsity of
$\Gamma=(\gamma_{ij})_{i,j}$.

\begin{assumption}[Availability of small-gain scalars for the terminal inequality]
\label{ass:terminal_comm}
At time $t$, subsystem $j$ broadcasts
\begin{equation}\label{eq:W_broadcast_scalar}
 s_j(t):=W_j(x_j(t))
\end{equation}
to every subsystem $i$ for which $\gamma_{ij}\not\equiv0$.  Equivalently, subsystem $i$ has
access to $\{s_j(t):j\in\mathcal G_i^\gamma\}$, where
$\mathcal G_i^\gamma:=\{j\in\{1,\ldots,\ell\}:\gamma_{ij}\not\equiv0\}$.
\end{assumption}

\begin{remark}
The local form
\begin{equation}\label{eq:terminal_ineq_local_max}
 W_i\bigl(x_i^{\nom}(M|t)\bigr)
 \leq\max_{j\in\mathcal G_i^\gamma}\gamma_{ij}(s_j(t))
\end{equation}
is identical to \eqref{eq:terminal_ineq_global_max} after zero gains are omitted.  Only scalar
values are exchanged on the $\gamma$-graph; full predicted trajectories remain restricted to
the physical-neighbor graph.
\end{remark}

\begin{remark}[Communication burden of Assumption~\ref{ass:terminal_comm}]
At each time step, Assumption~\ref{ass:terminal_comm} requires subsystem $j$ to communicate the scalar
$s_j$
to the subsystems that use it in \eqref{eq:terminal_ineq_global_max}. This channel is typically inexpensive compared with the prediction packets, whose size scales as ($\mathcal{O}(Mn_j)$), whereas $s_j$ consists of only one real number per subsystem and sampling instant.

If the gain matrix $\Gamma=(\gamma_{ij})$ is dense, the scalars may require network-wide dissemination. This can be implemented through a shared broadcast medium, a coordinator or aggregator, or multi-hop flooding or max-consensus over peer-to-peer links. Since only scalar quantities can be disseminated, the additional traffic is generally small relative to the exchanged prediction trajectories, although feasibility still depends on the sampling period, graph diameter, and communication latency.

Global dissemination is needed only for evaluating the right-hand side of the terminal inequality. When it is undesirable, one may instead use the local terminal condition \eqref{eq:terminal_ineq_local_max} or a distributed estimate of the required maximum. These alternatives reduce communication at the price of replacing exact global information by a local or approximate quantity.
\end{remark}

% ============================================================
%  COMMUNICATION ARCHITECTURE FOR DMPC
% ============================================================
The proposed DMPC scheme is independent of a particular communication topology. 
It only requires that the prediction packets in Assumption~\ref{ass:packets} and the scalar terminal information in Assumption~\ref{ass:terminal_comm} to be delivered within the timing prescribed by the control algorithm. 
Several standard communication network architectures can meet these requirements:

\paragraph{Coordinator-assisted architecture}
A gateway, aggregator, or supervisory controller collects the packets $\mathcal{P}_j(t)$, computes the required quantities~\eqref{eq:terminal_ineq_local_max}, and forwards them to the corresponding subsystems.

\paragraph{Peer-to-peer architecture}
Prediction packets are exchanged between physically coupled subsystems. The values $s_j$ as in~\eqref{eq:W_broadcast_scalar} are disseminated through flooding, gossip, or max-consensus, after which each subsystem evaluates~\eqref{eq:terminal_ineq_local_max}.

\paragraph{Hybrid hierarchical architecture}
Neighboring subsystems exchange prediction packets locally, while a supervisory node or cluster head collects the scalars $s_j$ as in~\eqref{eq:W_broadcast_scalar}, computes~\eqref{eq:terminal_ineq_local_max}, and returns the result to each subsystem.  

\begin{remark}
The controller depends only on Assumptions~\ref{ass:packets} and~\ref{ass:terminal_comm}, not on a specific communication realization. The choice among coordinator-assisted, P2P, and hybrid architectures is application-dependent and should reflect latency, reliability, cybersecurity, and available infrastructure~\cite{Scattolini.2009,Sandberg2015CyberphysicalSurvey}.
In the DC microgrid case-study, we follow a hybrid network architecture with a publisher/subscriber messaging pattern, see below for more details.
Apart from that, a detailed network design is outside the scope of this work.
\end{remark}

\subsection{Local distributed OCP with optimized tightening radii}
\label{subsec:ocp_online_tube}
For compact notation define
\begin{equation}\label{eq:radius_budget}
 \mathcal B_i(l|t)
 :=L_{i,2}\max_{j\in\Ni}\bar r_j(l|t)+d_i(l|t),
 \qquad l=0,\ldots,M-1.
\end{equation}
For $l=0,\ldots,M-2$, the packet-variation term computed from the exchanged predictions is
\begin{equation}\label{eq:di_def}
 d_i(l|t)
 :=L_{i,2}\max_{j\in\Ni}
 \bigl\|\bar x_j(l|t)-\bar x_j(l+1|t-1)\bigr\|.
\end{equation}
Both packet entries in \eqref{eq:di_def} predict the same absolute time.  At the final prediction
index there is no time-aligned entry $\bar x_j(M|t-1)$ in a length-$M$ packet.  We therefore
set
\begin{equation}\label{eq:terminal_packet_budget}
 d_i(M-1|t):=d_i^{\mathrm{ter}}(t),\qquad d_i^{\mathrm{ter}}(t)\geq0,
\end{equation}
where $d_i^{\mathrm{ter}}$ is a designer-supplied bound for the final packet-variation term; the admissible default
is $d_i^{\mathrm{ter}}(\cdot)=0$.
Here we formally state the local OCP. 

\begin{problem}\label{prob:OCP-2}
At $t\in\Zp$, let $M\in\N$, $M\geq2$, $\lambda>0$, and functions $W_i$ and $\gamma_{ij}$ satisfying
\eqref{eq:Wi-bounds} and \eqref{eq:SGC} be given.  Under
Assumptions~\ref{ass:packets} and~\ref{ass:terminal_comm}, compute
${\bf u}_{M,i}^\star(t)=\bigl(u_i^\star(0|t),\ldots,u_i^\star(M-1|t)\bigr)$ and
${\bf r}_{M-1,i}^\star(t)=\bigl(r_i^\star(1|t),\ldots,r_i^\star(M-1|t)\bigr)$
as a minimizer of
\begin{equation}\tag{OCP}\label{eq:OCP2}
\begin{aligned}
 \min_{{\bf u}_{M,i}(t),{\bf r}_{M-1,i}(t)}\quad
 &\sum_{l=0}^{M-1}W_i\bigl(x_i^{\nom}(l|t)\bigr)
   +\lambda\sum_{l=1}^{M-1}r_i(l|t)\\
 \mathrm{s.t.}\quad
 &x_i^{\nom}(0|t)=x_i(t),\\
 &x_i^{\nom}(l+1|t)
   =g_i\bigl(x_i^{\nom}(l|t),\bar x_{\Ni}(l|t),u_i(l|t)\bigr),
   &&l=0,\ldots,M-1,\\
 &u_i(l|t)\in\U_i,
   &&l=0,\ldots,M-1,\\
 &r_i(l|t)\geq0,
   &&l=1,\ldots,M-1,\\
 &r_i(l+1|t)\geq L_{i,1}r_i(l|t)+\mathcal B_i(l|t),
   &&l=0,\ldots,M-2,\\
 &x_i^{\nom}(l|t)\in\X_i\ominus\B_{r_i(l|t)},
   &&l=1,\ldots,M-1,\\
 &W_i\bigl(x_i^{\nom}(M|t)\bigr)
   \leq \max_{j\in\mathcal G_i^\gamma}\gamma_{ij}(s_j(t)) =: b_i(t).
\end{aligned}
\end{equation}
\end{problem}

In the OCP, $r_i(0|t)$ appearing in the propagation inequality for $l=0$ is the fixed value specified above. Hence, it is not a component of ${\bf r}_{M-1,i}(t)$.

After an optimizer has been selected, define the auxiliary terminal radius uniquely by
\begin{equation}\label{eq:terminal_radius_aux}
 r_i^\star(M|t)
 :=L_{i,1}r_i^\star(M-1|t)+\mathcal B_i(M-1|t).
\end{equation}
It is not an additional optimization variable and does not alter the implemented OCP.  Whenever an
optimizer is invoked below, existence of a minimizer of \eqref{eq:OCP2} is understood on the feasible
set under consideration.
Following a standard MPC implementation, for every $i$ and $t$, only the first component of the optimized input ${\bf u}_{M,i}^\star(t)$ obtained from~\eqref{eq:OCP2} is applied
\begin{equation}\label{eq:online_tube_law}
 u_i(t)=u_i^\star(0|t),\qquad i=1,\ldots,\ell.
\end{equation}
The complete parallel routine is as follows.
\begin{algorithm}[H]
\caption{DMPC routine with optimized tightening radii}
\label{alg:dist_multi_step_mpc}
\begin{algorithmic}[1]
\State Fix $M\in\N$, $M\geq2$, $\lambda>0$, and initialize feasible packets.
\For{$t=0,1,2,\ldots$}
\State Receive $\{\mathcal P_j(t)\}_{j\in\Ni}$ and the required scalars $s_j(t)$.
\State Compute \eqref{eq:di_def} for $l=0,\ldots,M-2$ and choose
$d_i^{\mathrm{ter}}(t)$ in \eqref{eq:terminal_packet_budget}.
\State Solve \eqref{eq:OCP2} for
$({\bf u}_{M,i}^\star(t),{\bf r}_{M-1,i}^\star(t))$.
\State Compute the auxiliary $r_i^\star(M|t)$ from \eqref{eq:terminal_radius_aux}.
\State Apply $u_i(t)=u_i^\star(0|t)$ to \eqref{eq:network_decoupled}.
\State Measure $x_i(t+1)$, form the shifted packet using
\eqref{eq:packet_shift_x}--\eqref{eq:packet_shift_r}, and broadcast it.
\EndFor
\end{algorithmic}
\end{algorithm}

\section{Feasibility and Stability Guarantees}\label{sec:dist-fsclf-stability}
The recursive-feasibility and stability arguments require different consistency properties.
Recursive feasibility is a property of the \emph{new nominal OCP} after the packets and the
measured initial state have changed.  Stability additionally requires a comparison between the
actual repeatedly reoptimized trajectory and an earlier full-horizon nominal prediction.  These
two comparisons are kept separate below.

\subsection{Recursive feasibility by shift-compatible margins and one-step terminal feasibility}
\label{subsec:online_tube_rec_feas}
Fix a feasible solution of \eqref{eq:OCP2} at time $t$.  At time $t+1$, use the shifted prefix
\begin{equation}\label{eq:shifted_input_candidate}
 \tilde u_i(l|t+1):=u_i^\star(l+1|t),\qquad l=0,\ldots,M-2.
\end{equation}
Let $\tilde x_i^{\nom}(\cdot|t+1)$ denote the nominal trajectory generated from
$\tilde x_i^{\nom}(0|t+1)=x_i(t+1)$, the new packets, and the prefix
\eqref{eq:shifted_input_candidate}.
Take the minimum candidate radii by $\tilde r_i(0|t+1):=0,$
\begin{equation}
\label{eq:candidate_radius_recursion}
 \tilde r_i(l+1|t+1):=L_{i,1}\tilde r_i(l|t+1)+\mathcal B_i(l|t+1),
 \quad l=0,\ldots,M-2,
\end{equation}
and set $\tilde r_i(M|t+1)$ analogously to~\eqref{eq:terminal_radius_aux}.
The difference between the new shifted nominal prefix and the old shifted nominal trajectory is
bounded by the computable sequence $\chi_i(0|t+1):=0$,
\begin{equation}
\label{eq:chi_shift_recursion}
 \chi_i(l+1|t+1):=L_{i,1}\chi_i(l|t+1)
 +L_{i,2}\bigl\|\bar x_{\Ni}(l|t+1)-\bar x_{\Ni}(l+1|t)\bigr\|_{\Ni,\infty},
\end{equation}
for $l=0,\ldots,M-2$.
Indeed, the shifted candidate and the old prediction use the same local input at every prefix
stage, so \eqref{eq:Lipschitz_dynamics_p} and induction give
\begin{equation}\label{eq:nominal_shift_error}
 \bigl\|\tilde x_i^{\nom}(l|t+1)-x_i^{\nom,\star}(l+1|t)\bigr\|
 \leq\chi_i(l|t+1),\qquad l=0,\ldots,M-1.
\end{equation}
For $l=1,\ldots,M-1$, define the total shift margin
\begin{equation}\label{eq:shift_total_margin}
 \delta_i(l|t+1):=\chi_i(l|t+1)+\tilde r_i(l|t+1).
\end{equation}
The following condition states directly, using the Pontryagin-difference notation introduced above, that the old shifted nominal state has enough state-constraint margin to
accommodate both the shift error \eqref{eq:nominal_shift_error} and the new tightening radius.

\begin{assumption}\label{ass:shift_margin_condition}
At every transition from a feasible OCP at time $t$ to the candidate at time $t+1$,
\begin{equation}\label{eq:shift_margin_condition}
 x_i^{\nom,\star}(l+1|t)
 \in \X_i\ominus\B_{\delta_i(l|t+1)},
 \qquad l=1,\ldots,M-1.
\end{equation}
Equivalently,
$x_i^{\nom,\star}(l+1|t)+\B_{\delta_i(l|t+1)}\subseteq\X_i$.
\end{assumption}

The next lemma gives directly checkable forms of \eqref{eq:shift_margin_condition}.
\begin{lemma}
\label{lem:shift_margin_verification}
Let $\X_i=\{z\in\R^{n_i}:F_i z\leq f_i\}$ be a nonempty polyhedron and assume each row
$F_{i,k}$ is nonzero.  Then \eqref{eq:shift_margin_condition} holds at prediction index $l$ if and
only if, for every row $k$,
\begin{equation}\label{eq:shift_margin_polyhedral_certificate}
 F_{i,k}x_i^{\nom,\star}(l+1|t)
 +\|F_{i,k}\|_2\,\delta_i(l|t+1)
 \leq f_{i,k}.
\end{equation}
For a box $\X_i=\{z:\underline x_i\leq z\leq\overline x_i\}$, this is equivalent to
\begin{equation}\label{eq:shift_margin_box_certificate}
 \underline x_i+\delta_i(l|t+1)\mathbf 1
 \leq x_i^{\nom,\star}(l+1|t)
 \leq\overline x_i-\delta_i(l|t+1)\mathbf 1.
\end{equation}
If, moreover, $x_i^{\nom,\star}(l+1|t)$ belongs to
$\mathcal E(c,Q):=\{c+Q^{1/2}w:\|w\|_2\leq1\}$, $Q\succeq0$, and
$\delta_i(l|t+1)\leq\bar\delta$, then
\begin{equation}\label{eq:shift_margin_ellipsoid_certificate}
 F_{i,k}c+\sqrt{F_{i,k}QF_{i,k}^\top}
 +\|F_{i,k}\|_2\bar\delta\leq f_{i,k}
\end{equation}
for every row $k$ is a sufficient uniform certificate.
\end{lemma}

\begin{proof}
Let $z:=x_i^{\nom,\star}(l+1|t)$ and $\delta:=\delta_i(l|t+1)$.  By the definition of the
Pontryagin difference,
\[
 z\in\X_i\ominus\B_\delta
 \quad\Longleftrightarrow\quad
 F_{i,k}(z+e)\leq f_{i,k}
 \quad\forall\,\|e\|_2\leq\delta,\ \forall k.
\]
For each row,
$\max_{\|e\|_2\leq\delta}F_{i,k}e=\delta\|F_{i,k}\|_2$ by the dual-norm formula.
Thus the preceding set inclusion is equivalent row by row to
\eqref{eq:shift_margin_polyhedral_certificate}.  Applying these inequalities to the $2n_i$
half-spaces defining a box gives \eqref{eq:shift_margin_box_certificate}.  Finally,
$\max_{z\in\mathcal E(c,Q)}F_{i,k}z=F_{i,k}c+\sqrt{F_{i,k}QF_{i,k}^\top}$; combining this
support function with $\delta_i\leq\bar\delta$ gives
\eqref{eq:shift_margin_ellipsoid_certificate}.
\end{proof}

\begin{remark}
For polyhedral or box constraints, \eqref{eq:shift_margin_polyhedral_certificate} or
\eqref{eq:shift_margin_box_certificate} can be checked online.  Condition
\eqref{eq:shift_margin_ellipsoid_certificate} is a sufficient offline certificate over a prescribed
ellipsoidal envelope.  For a general closed state set, the mathematically equivalent condition remains
the set inclusion in Assumption~\ref{ass:shift_margin_condition}; its computational treatment depends
on the representation of $\X_i$.
\end{remark}

It remains to verify the terminal inequality for the last input of the shifted candidate.
\begin{problem}[One-step terminal feasibility problem]\label{prob:one_step_completion_problem}
Suppose that the local OCPs are feasible at time $t$.  At time $t+1$, with the measured network
state $x(t+1)$, find $v_i\in\U_i$ such that
\begin{equation}\label{eq:one_step_completion_problem}
 W_i\!\left(g_i\!\left(\tilde x_i^{\nom}(M-1|t+1),
 \bar x_{\Ni}(M-1|t+1),v_i\right)\right)
 \leq \max_{j\in\mathcal G_i^\gamma}\gamma_{ij}\bigl(W_j(x_j(t+1))\bigr).
\end{equation}
\end{problem}
For subsystems that are affine in the control input, this one-step condition has standard convex
representations.

\begin{lemma}[Tractable one-step terminal feasibility tests]
\label{lem:tractable_completion}
Let $z_i:=\tilde x_i^{\nom}(M-1|t+1)$,
$\zeta_{\Ni}:=\bar x_{\Ni}(M-1|t+1)$, and
$b:=\max_{j\in\mathcal G_i^\gamma}\gamma_{ij}(W_j(x_j(t+1)))$.
\begin{enumerate}[(i)]
\item Suppose
$g_i(z_i,\zeta_{\Ni},v_i)=a_i(z_i,\zeta_{\Ni})+G_i v_i$,
$W_i(y)=y^\top P_i y$ with $P_i\succ0$, and
$\U_i=\{v_i:H_i v_i\leq h_i\}$.  Then \eqref{eq:one_step_completion_problem} is equivalent to
\begin{equation}\label{eq:completion_socp}
 H_i v_i\leq h_i,\qquad
 \|P_i^{1/2}(a_i(z_i,\zeta_{\Ni})+G_i v_i)\|_2\leq\sqrt b,
\end{equation}
which is an SOCP feasibility problem.
\item If $v_i$ is scalar, $\U_i=[\ul u_i,\ol u_i]$, and
$W_i(y)=p_i(c_i^\top y)^2$, set
$\alpha_i:=c_i^\top a_i(z_i,\zeta_{\Ni})$,
$\beta_i:=c_i^\top G_i$, and $h_i:=\sqrt{b/p_i}$.  If $\beta_i\neq0$, the condition is feasible
if and only if
\begin{equation}\label{eq:scalar_completion_interval}
 [\ul u_i,\ol u_i]\cap
 \left[\min\left\{\frac{-h_i-\alpha_i}{\beta_i},
                       \frac{ h_i-\alpha_i}{\beta_i}\right\},
       \max\left\{\frac{-h_i-\alpha_i}{\beta_i},
                       \frac{ h_i-\alpha_i}{\beta_i}\right\}\right]
 \neq\varnothing.
\end{equation}
If $\beta_i=0$, it is feasible if and only if $|\alpha_i|\leq h_i$.
\end{enumerate}
\end{lemma}

\begin{proof}
Under (i), the inequality $W_i(g_i)\leq b$ is exactly the Euclidean-norm inequality in
\eqref{eq:completion_socp}; the input constraints are polyhedral.  Under (ii), the same inequality is
$|\alpha_i+\beta_i v_i|\leq h_i$.  Solving these two affine inequalities and intersecting the resulting
interval with $\U_i$ gives \eqref{eq:scalar_completion_interval}.
\end{proof}

\begin{theorem}[Recursive feasibility]
\label{thm:online_tube_rec_feas}
Consider \eqref{eq:network} and Problem~\ref{prob:OCP-2}, and let
Assumptions~\ref{ass:packets} and~\ref{ass:terminal_comm} hold.  Suppose that the local OCPs are
feasible at time $t$, Assumption~\ref{ass:shift_margin_condition} holds for the transition to $t+1$,
and Problem~\ref{prob:one_step_completion_problem} is feasible for every subsystem at that transition.
Then the local OCPs are feasible at time $t+1$.  Consequently, if these conditions hold at every
transition and the OCPs are feasible at $t=0$, then the OCPs are recursively feasible for all
$t\in\Zp$.
\end{theorem}

\begin{proof}
Fix subsystem $i$ and a feasible optimizer at time $t$.  Since
$\bar x_{\Ni}(0|t)=x_{\Ni}(t)$ and the applied input is $u_i^\star(0|t)$,
$
 x_i(t+1)=x_i^{\nom,\star}(1|t)
$.
Hence the new measured initial state equals the first old nominal successor.  Choose the first
$M-1$ candidate inputs according to \eqref{eq:shifted_input_candidate}; these inputs remain in
$\U_i$.  Choose the candidate radii from \eqref{eq:candidate_radius_recursion}; therefore all radius
propagation and nonnegativity constraints are satisfied.

Fix $l\in\{1,\ldots,M-1\}$ and any $e_i$ with
$\|e_i\|\leq\tilde r_i(l|t+1)$.  Put
$
 d_i:=\tilde x_i^{\nom}(l|t+1)-x_i^{\nom,\star}(l+1|t).
$
By \eqref{eq:nominal_shift_error}, $\|d_i\|\leq\chi_i(l|t+1)$, and hence
$
 \|d_i+e_i\|\leq\chi_i(l|t+1)+\tilde r_i(l|t+1)
 =\delta_i(l|t+1)$.
Assumption~\ref{ass:shift_margin_condition} is equivalent to
$x_i^{\nom,\star}(l+1|t)+\B_{\delta_i(l|t+1)}\subseteq\X_i$; consequently
\[
 \tilde x_i^{\nom}(l|t+1)+e_i
 =x_i^{\nom,\star}(l+1|t)+(d_i+e_i)\in\X_i.
\]
Since $e_i$ was arbitrary,
$\tilde x_i^{\nom}(l|t+1)\in\X_i\ominus\B_{\tilde r_i(l|t+1)}$ for every
$l=1,\ldots,M-1$.

By feasibility of Problem~\ref{prob:one_step_completion_problem}, choose
$\tilde u_i(M-1|t+1)\in\U_i$ satisfying \eqref{eq:one_step_completion_problem}.  Appending this
input to the shifted prefix supplies the terminal inequality of the new OCP.  Thus every constraint in
\eqref{eq:OCP2} is satisfied at time $t+1$.
\end{proof}

Repeated application gives the following consequence.
\begin{corollary}[Recursive feasibility and constraint satisfaction]
\label{cor:true_constraint_satisfaction}
Under Theorem~\ref{thm:online_tube_rec_feas}, if $x(0)\in\X$, then
$x(t)\in\X$ and $u(t)\in\U$ for all $t\in\Zp$.
\end{corollary}

\subsection{Stability guarantees}
\label{subsec:online_tube_stability}
We now quantify the difference between the actual repeatedly reoptimized trajectory and the
full nominal prediction computed at a fixed time $t$.  Define 
 $p_i(l|t) :=\|x_{\Ni}(t+l)-\bar x_{\Ni}(l|t)\|_{\Ni,\infty}$, $\delta_i^u(l|t) :=\|u_i^\star(0|t+l)-u_i^\star(l|t)\|$, for $l=0,\ldots,M-1$.
 Both quantities are measurable retrospectively from stored packets,
state measurements, and optimizer outputs.  In particular, $p_i(0|t)=0$ and
$\delta_i^u(0|t)=0$.
The portion of the actual mismatch not accounted for by the forcing term \eqref{eq:radius_budget} is
\begin{equation}\label{eq:eta_def}
 \eta_i(l|t)
 :=\max\Bigl\{ L_{i,2}p_i(l|t)+L_{i,3}\delta_i^u(l|t)-\mathcal B_i(l|t),0\Bigr\}.
\end{equation}
Define the additional analytical mismatch bound, which is not an OCP variable, by
\begin{equation}
\label{eq:s_correction_recursion} 
 s_i(l+1|t) :=L_{i,1}s_i(l|t)+\eta_i(l|t),
\end{equation}
with $s_i(0|t):=0$, $l=0,\ldots,M-1$.

\begin{lemma}
\label{lem:corrected_actual_nominal_bound}
Let the dynamics $g_i$ of $\Sigma_i$ satisfy the Lipschitz continuity~\eqref{eq:Lipschitz_dynamics_p}.
For every feasible closed-loop execution and
all $l=0,\ldots,M$,
\begin{equation}\label{eq:corrected_actual_nominal_bound}
 \|x_i(t+l)-x_i^{\nom,\star}(l|t)\|
 \leq r_i^\star(l|t)+s_i(l|t).
\end{equation}
For $l=M$, $r_i^\star(M|t)$ is the auxiliary value from
\eqref{eq:terminal_radius_aux}.
\end{lemma}

\begin{proof}
Fix $t$ and $i$, and define $e_i(l|t):=\|x_i(t+l)-x_i^{\nom,\star}(l|t)\|$, $l=0,\ldots,M$.
At prediction step $l\leq M-1$, the actual and time-$t$ nominal successors are, respectively,
 $x_i(t+l+1)
 =g_i\bigl(x_i(t+l),x_{\Ni}(t+l),u_i^\star(0|t+l)\bigr)$, $x_i^{\nom,\star}(l+1|t)
 =g_i\bigl(x_i^{\nom,\star}(l|t),\bar x_{\Ni}(l|t),u_i^\star(l|t)\bigr)$.
Applying \eqref{eq:Lipschitz_dynamics_p} to these two arguments gives
\begin{equation}\label{eq:e_error_recursion}
 e_i(l+1|t)
 \leq L_{i,1}e_i(l|t)+L_{i,2}p_i(l|t)+L_{i,3}\delta_i^u(l|t).
\end{equation}
To relate the last two terms to the radius recursion, set $a_i(l|t):=L_{i,2}p_i(l|t)+L_{i,3}\delta_i^u(l|t)\geq 0$.
By \eqref{eq:eta_def},
$\eta_i(l|t)=\max\{a_i(l|t)-\mathcal B_i(l|t),0\}$.  Therefore
\begin{equation}\label{eq:eta_dominates_excess}
 a_i(l|t)\leq\mathcal B_i(l|t)+\eta_i(l|t).
\end{equation}
Indeed, if $a_i\leq\mathcal B_i$, then \eqref{eq:eta_dominates_excess} follows from
$\eta_i=0$; if $a_i>\mathcal B_i$, then
$\eta_i=a_i-\mathcal B_i$ and equality holds.  This is the precise role of the maximum in
\eqref{eq:eta_def}.

At $l=0$, the measured initial state is also the nominal initial state, so
$e_i(0|t)=0=r_i^\star(0|t)+s_i(0|t)$.  Assume for some
$l\in\{0,\ldots,M-1\}$ that
$e_i(l|t)\leq r_i^\star(l|t)+s_i(l|t)$.  Combining
\eqref{eq:e_error_recursion} and \eqref{eq:eta_dominates_excess} yields
$e_i(l+1|t)
\leq L_{i,1}\bigl(r_i^\star(l|t)+s_i(l|t)\bigr)
      +\mathcal B_i(l|t)+\eta_i(l|t)
=\bigl(L_{i,1}r_i^\star(l|t)+\mathcal B_i(l|t)\bigr)
   +s_i(l+1|t)$.
For $l=0,\ldots,M-2$, the OCP radius inequality gives
$L_{i,1}r_i^\star(l|t)+\mathcal B_i(l|t)\leq r_i^\star(l+1|t)$.
For $l=M-1$, the same expression equals $r_i^\star(M|t)$ by
\eqref{eq:terminal_radius_aux}.  Hence in both cases
$e_i(l+1|t)\leq r_i^\star(l+1|t)+s_i(l+1|t)$.
Induction over $l=0,\ldots,M-1$ proves \eqref{eq:corrected_actual_nominal_bound} for the entire
finite-step interval, including its terminal point.
\end{proof}

\begin{proposition}[$M$-step practical Lyapunov estimate under prediction and reoptimization mismatch]
\label{prop:prac_asyp_like_decay}
Suppose Assumption~\ref{ass:clf-small-gain} holds, the local dynamics satisfy
\eqref{eq:Lipschitz_dynamics_p}, and the OCPs are recursively feasible.  Let
$\mathcal Z\subset\R^n$ be a compact set containing the actual closed-loop states $x(t)$ and all
nominal prediction vectors $x^{\nom,\star}(l|t)$, $l=0,\ldots,M$, for the times under consideration.  Define
\begin{equation}\label{eq:rho_terminal_total}
 \rho(t):=\max_{i=1,\ldots,\ell}
 \bigl(r_i^\star(M|t)+s_i(M|t)\bigr).
\end{equation}
Then the closed-loop satisfies the constraints of Corollary~\ref{cor:true_constraint_satisfaction}.
Moreover, suitable small-gain scaling functions $\sigma_i\in\Kinf$ define
\begin{equation}\label{eq:Wconstruction}
 V(x):=\max_{i=1,\ldots,\ell}\sigma_i^{-1}(W_i(x_i)),
\end{equation}
for which there exist $\ul\alpha_V,\ol\alpha_V,\omega_V\in\Kinf$ satisfying
\begin{align}
\label{eq:V-bounds}
 \ul\alpha_V(\|x\|)&\leq V(x)\leq\ol\alpha_V(\|x\|),\\
\label{eq:V_modulus}
 |V(z)-V(z')|&\leq\omega_V(\|z-z'\|),\qquad z,z'\in\mathcal Z.
\end{align}
In addition, there exist $\widehat\alpha,\widehat\gamma\in\Kinf$, with
$\widehat\alpha<\id$, such that
\begin{equation}\label{eq:pract_asym_sta_like}
 V(x(t+M))
 \leq\max\bigl\{\widehat\alpha(V(x(t))),\widehat\gamma(\rho(t))\bigr\},
 \qquad t\in\Zp.
\end{equation}
Thus \eqref{eq:V_modulus} is a consequence of the continuity of the local functions and compactness
of the operating set; it is not an independent assumption on a function introduced later.
\end{proposition}

\begin{proof}
The cyclic small-gain condition gives scaling functions $\sigma_i\in\Kinf$ such that
\[
 \max_{j\in\mathcal G_i^\gamma}
 \sigma_i^{-1}\circ\gamma_{ij}\circ\sigma_j<\id,
 \qquad i=1,\ldots,\ell;
\]
see, e.g., \cite[Theorem~5.5]{Ruffer.2010}.
Define
$ \alpha_{\mathrm{sg}}
 :=\max_{\substack{i=1,\ldots,\ell\\j\in\mathcal G_i^\gamma}}
 \sigma_i^{-1}\circ\gamma_{ij}\circ\sigma_j$.
Then $\alpha_{\mathrm{sg}}\in\Kinf$ and $\alpha_{\mathrm{sg}}<\id$.

We first establish the comparison bounds and continuity of $V$.  Since every $W_i$ is continuous by
Assumption~\ref{ass:clf-small-gain}, each map $x_i\mapsto\sigma_i^{-1}(W_i(x_i))$ is continuous;
the maximum of finitely many continuous functions is therefore continuous.  From
\eqref{eq:Wi-bounds}, define
\begin{align*}
 \ul\alpha_V(s)&:=\min_{i=1,\ldots,\ell}
 \sigma_i^{-1}\!\left(\ul\alpha_i\!\left(\frac{s}{\sqrt\ell}\right)\right),\\
 \ol\alpha_V(s)&:=\max_{i=1,\ldots,\ell}
 \sigma_i^{-1}\bigl(\ol\alpha_i(s)\bigr).
\end{align*}
Both functions belong to $\Kinf$.  Since
$\max_i\|x_i\|\geq\|x\|/\sqrt\ell$ and $\|x_i\|\leq\|x\|$, these definitions give
\eqref{eq:V-bounds}.

We next prove \eqref{eq:V_modulus}.
For $s\geq0$ set
$\nu(s):=\max\bigl\{|V(z)-V(z')|:\ z,z'\in\mathcal Z,\ \|z-z'\|\leq s\bigr\}$.
The maximum exists because the admissible pair set is compact.  The function $\nu$ is finite and
nondecreasing, $\nu(0)=0$, and uniform continuity of $V$ on the compact set $\mathcal Z$ implies
$\nu(s)\to0$ as $s\downarrow0$.  Define
\begin{equation}\label{eq:omegaV_construction}
 \omega_V(s):=s+\int_1^2\nu(\tau s)\,d\tau,
 \qquad s\geq0.
\end{equation}
Because $\nu$ is nondecreasing, the integral term is nondecreasing.  Its continuity in $s$ follows from dominated convergence (a monotone function has at most countably many discontinuities), and the additional term $s$ makes $\omega_V$ strictly increasing and unbounded.  Hence
$\omega_V\in\Kinf$.  Moreover, $\nu(\tau s)\geq\nu(s)$ for $\tau\in[1,2]$, so
$\omega_V(s)\geq\nu(s)$, which proves \eqref{eq:V_modulus}.

It remains to establish the finite-step estimate.  From \eqref{eq:Wconstruction},
$W_j(x_j(t))\leq\sigma_j(V(x(t)))$.  Therefore the local terminal inequalities give
\begin{align}
\label{eq:nominal_terminal_decay}
 V(x^{\nom,\star}(M|t))
 &\leq\max_i\max_{j\in\mathcal G_i^\gamma}
 \sigma_i^{-1}\!\left(\gamma_{ij}(W_j(x_j(t)))\right)\nonumber\\
 &\leq\alpha_{\mathrm{sg}}(V(x(t))).
\end{align}
By Lemma~\ref{lem:corrected_actual_nominal_bound} and the definition of $\rho$,
 $\|x(t+M)-x^{\nom,\star}(M|t)\|
 \leq\left(\sum_{i=1}^{\ell}(r_i^\star(M|t)+s_i(M|t))^2\right)^{1/2}
 \leq\sqrt\ell\,\rho(t)$.
Using \eqref{eq:V_modulus} and \eqref{eq:nominal_terminal_decay},
\begin{equation}\label{eq:additive_finite_step_estimate}
 V(x(t+M))
 \leq\alpha_{\mathrm{sg}}(V(x(t)))
 +\omega_V(\sqrt\ell\,\rho(t)).
\end{equation}

Let $\bar V:=\max_{z\in\mathcal Z}V(z)$.  Since $\alpha_{\mathrm{sg}}(s)<s$ for $s>0$, one may
choose $\varrho\in\Kinf$ sufficiently small on the compact interval
$[0,\alpha_{\mathrm{sg}}(\bar V)]$ so that
\begin{equation}\label{eq:rho_weak_triangle_choice}
 (\id+\varrho)\circ\alpha_{\mathrm{sg}}(s)<s,
 \qquad s\in(0,\bar V].
\end{equation}
Put $c:=(\id+\varrho)\circ\alpha_{\mathrm{sg}}(\bar V)<\bar V$ and choose any
$\theta\in(0,1)$.  Define $\widehat\alpha=(\id+\varrho)\circ\alpha_{\mathrm{sg}}$ on
$[0,\bar V]$ and, for $s>\bar V$, set
$\widehat\alpha(s):=c+\theta(s-\bar V)$.  This gives
$\widehat\alpha\in\Kinf$ and $\widehat\alpha<\id$ on $\Rp$.  For $a,b\geq0$,
$a+b\leq\max\{(\id+\varrho)(a),(\id+\varrho^{-1})(b)\}$.
Applying this inequality to \eqref{eq:additive_finite_step_estimate} and defining
$\widehat\gamma:=(\id+\varrho^{-1})\circ\omega_V\circ(\sqrt\ell\,\id)$
yields \eqref{eq:pract_asym_sta_like}.
\end{proof}

\begin{lemma}[Constraint satisfaction and practical bound at intermediate time instants]
\label{lem:initial_intermediate_K_bound}
Let the hypotheses of Proposition~\ref{prop:prac_asyp_like_decay} hold, and let
$q:\X\to\U^M$ be the admissible finite-step feedback from
Assumption~\ref{ass:clf-small-gain}.  Fix any $\tau\in\Zp$ and put $\xi_\tau:=x(\tau)$.
Let $J_i^\star(\tau)$ denote the optimal value of the $i$th local OCP
\eqref{eq:OCP2} at time $\tau$.  Let $\kappa_0^q:=\id$ and let
$\kappa_l^q\in\K$, $l=1,\ldots,M$, be the comparison functions supplied by
admissibility of $q$, i.e.,
$\|x(l,\xi,u_q)\|\leq\kappa_l^q(\|\xi\|)$, $l=0,\ldots,M$.
Define
\begin{equation}\label{eq:psi_q_intermediate}
 \psi_i(s):=\sum_{l=0}^{M-1}\ol\alpha_i(\kappa_l^q(s)),
 \qquad i=1,\ldots,\ell,
\end{equation}
and the nonnegative excess of the local optimal cost over this admissible-feedback
benchmark by
\begin{equation}\label{eq:deltaJ_intermediate}
 \delta_{J,i}(\tau)
 :=\max\bigl\{J_i^\star(\tau)-\psi_i(\|\xi_\tau\|),0\bigr\}.
\end{equation}
Furthermore, define the actual-to-time-$\tau$ nominal mismatch over the intermediate
indices of this finite-step interval by
\begin{equation}\label{eq:epsilon_intermediate}
 \varepsilon(\tau)
 :=\max_{r=0,\ldots,M-1}
 \left(
 \sum_{i=1}^{\ell}
 \bigl(r_i^\star(r|\tau)+s_i(r|\tau)\bigr)^2
 \right)^{1/2},
\end{equation}
and set
\begin{align}
\label{eq:DeltaJ_intermediate}
 \Delta_J(\tau)
 &:=\max_{i=1,\ldots,\ell}
 \sigma_i^{-1}\!\bigl(2\delta_{J,i}(\tau)\bigr),\\
\label{eq:Delta_intermediate}
 \Delta(\tau)
 &:=\Delta_J(\tau)+\omega_V(\varepsilon(\tau)).
\end{align}
Then there exists $\kappa_{\mathrm{int}}\in\Kinf$, independent of $\tau$, such that
for every $r=0,\ldots,M-1$,
\begin{align}
\label{eq:intermediate_state_constraint}
 x_i(\tau+r)&\in\X_i,\qquad i=1,\ldots,\ell,\\
\label{eq:initial_intermediate_K_bound}
 V(x(\tau+r))
 &\leq\kappa_{\mathrm{int}}(V(x(\tau)))+\Delta(\tau).
\end{align}
\end{lemma}

\begin{proof}
We first prove \eqref{eq:intermediate_state_constraint}.  For $r=0$ the assertion
follows from $\xi_\tau=x(\tau)\in\X$, which is guaranteed by recursive feasibility.
For $r\geq1$, recursive feasibility gives a feasible local OCP at time
$\tau+r-1$.  By Assumption~\ref{ass:packets},
$\bar x_{\Ni}(0|\tau+r-1)=x_{\Ni}(\tau+r-1)$, and by
\eqref{eq:online_tube_law} the input applied to the plant is
$u_i^\star(0|\tau+r-1)$.  Hence the actual successor and the first nominal
successor coincide:
$x_i(\tau+r)
 =
 g_i\!\left(
 x_i(\tau+r-1),x_{\Ni}(\tau+r-1),
 u_i^\star(0|\tau+r-1)
 \right)
 =x_i^{\nom,\star}(1|\tau+r-1)$.
The stage-$1$ state constraint of \eqref{eq:OCP2} gives
$ x_i^{\nom,\star}(1|\tau+r-1)
 \in \X_i\ominus\B_{r_i^\star(1|\tau+r-1)}
 \subseteq\X_i$,
and therefore $x_i(\tau+r)\in\X_i$.
This argument applies successively to
all intermediate physical time instants.  Notice that this constraint-satisfaction
argument uses recursive feasibility and the exact first-step initialization of each
receding-horizon OCP; it does not require the additional analytical mismatch bound
$s_i$ to be included in the OCP tightening.

We next prove \eqref{eq:initial_intermediate_K_bound}.  Since every term in the
local OCP objective is nonnegative, for every $r=0,\ldots,M-1$,
\begin{equation}\label{eq:Wi_nominal_by_Jstar}
 W_i(x_i^{\nom,\star}(r|\tau))\leq J_i^\star(\tau).
\end{equation}
By \eqref{eq:deltaJ_intermediate}, we get $J_i^\star(\tau)
 \leq \psi_i(\|\xi_\tau\|)+\delta_{J,i}(\tau)$.
Using $a+b\leq\max\{2a,2b\}$ for $a,b\geq0$ and monotonicity of
$\sigma_i^{-1}$, \eqref{eq:Wi_nominal_by_Jstar} implies
$\sigma_i^{-1}\!\left(W_i(x_i^{\nom,\star}(r|\tau))\right)
 \leq
 \max\left\{
 \sigma_i^{-1}\!\left(2\psi_i(\|\xi_\tau\|)\right),
 \sigma_i^{-1}\!\left(2\delta_{J,i}(\tau)\right)
 \right\}$.
From the lower comparison bound in \eqref{eq:V-bounds},
$\|\xi_\tau\|\leq\ul\alpha_V^{-1}(V(\xi_\tau))$.  Thus, with
\begin{equation}\label{eq:kappa_int_construction}
 \kappa_{\mathrm{int}}(s)
 :=
 \max_{i=1,\ldots,\ell}
 \sigma_i^{-1}\!\left(
 2\psi_i\bigl(\ul\alpha_V^{-1}(s)\bigr)
 \right),
 \qquad s\geq0,
\end{equation}
we obtain
\begin{equation}\label{eq:nominal_intermediate_practical_bound}
 V(x^{\nom,\star}(r|\tau))
 \leq
 \max\left\{
 \kappa_{\mathrm{int}}(V(x(\tau))),
 \Delta_J(\tau)
 \right\}.
\end{equation}
The term with $l=0$ in \eqref{eq:psi_q_intermediate} is
$\ol\alpha_i(s)$; consequently each $\psi_i\in\Kinf$, and
\eqref{eq:kappa_int_construction} gives
$\kappa_{\mathrm{int}}\in\Kinf$.

By Lemma~\ref{lem:corrected_actual_nominal_bound},
\[
 \|x(\tau+r)-x^{\nom,\star}(r|\tau)\|
 \leq
 \left(
 \sum_{i=1}^{\ell}
 \bigl(r_i^\star(r|\tau)+s_i(r|\tau)\bigr)^2
 \right)^{1/2}
 \leq\varepsilon(\tau).
\]
Both vectors belong to the compact set $\mathcal Z$ specified in
Proposition~\ref{prop:prac_asyp_like_decay}; hence
\eqref{eq:V_modulus} gives
$ V(x(\tau+r))
 \leq
 V(x^{\nom,\star}(r|\tau))
 +\omega_V(\varepsilon(\tau))$.
Combining this inequality with
\eqref{eq:nominal_intermediate_practical_bound} and using
$\max\{a,b\}\leq a+b$ yields
 $V(x(\tau+r))
 \leq
 \kappa_{\mathrm{int}}(V(x(\tau)))
 +\Delta_J(\tau)+\omega_V(\varepsilon(\tau))$,
which is exactly \eqref{eq:initial_intermediate_K_bound} with
\eqref{eq:Delta_intermediate}.

The offset $\Delta(\tau)$ has two distinct sources.  The term
$\Delta_J(\tau)$ measures the finite excess of the actual local OCP optimal values
over the admissible-feedback comparison bound, and therefore includes the effect of
nonzero tightening radii and packet-dependent feasibility restrictions on the local
optimization.  The term $\omega_V(\varepsilon(\tau))$ is the Lyapunov-function
effect of the actual-to-nominal mismatch quantified by
Lemma~\ref{lem:corrected_actual_nominal_bound}.  Any additional uncertainty not
already represented in that actual-to-nominal error bound would require a
corresponding additional term in Lemma~\ref{lem:corrected_actual_nominal_bound};
no claim is made here for arbitrary unmodelled disturbances.
\end{proof}

\begin{remark}
If, on a prescribed set of initial conditions or operating points, one has verified
uniform bounds
$\delta_{J,i}(\tau)\leq\bar\delta_{J,i}$,
$\varepsilon(\tau)\leq\bar\varepsilon$,
then \eqref{eq:initial_intermediate_K_bound} holds uniformly with the constant
$ \bar\Delta
 :=
 \max_i\sigma_i^{-1}(2\bar\delta_{J,i})
 +\omega_V(\bar\varepsilon)$.
Thus a strictly positive offset is permitted.  Such an offset gives a practical
intermediate-time estimate.  To recover Lyapunov stability of the origin, the
offset must additionally vanish with the initial state, as made precise in
Proposition~\ref{prop:practical_form_rigorous}.
\end{remark}

\begin{proposition}[Practical boundedness, convergence, and asymptotic stability]
\label{prop:practical_form_rigorous}
Let the hypotheses of Proposition~\ref{prop:prac_asyp_like_decay} hold.  For each intermediate index
$r\in\{0,\ldots,M-1\}$ define
$Y_r(k):=V(x(r+kM))$ and $R_r(k):=\rho(r+kM)$.  Then, for every $k\in\Zp$,
\begin{equation}\label{eq:intermediate_iteration_exact}
 Y_r(k)
 \leq\max\left\{
 \widehat\alpha^{(k)}(Y_r(0)),
 \max_{q=0,\ldots,k-1}
 \widehat\alpha^{(k-1-q)}\circ\widehat\gamma(R_r(q))
 \right\},
\end{equation}
where the second maximum is zero for $k=0$.  Let
$\Delta_0:=\Delta(0)$ be the offset from
Lemma~\ref{lem:initial_intermediate_K_bound}.  Then
\begin{equation}\label{eq:intermediate_iteration_with_offset}
 Y_r(k)
 \leq\max\left\{
 \widehat\alpha^{(k)}
 \!\left(\kappa_{\mathrm{int}}(V(x(0)))+\Delta_0\right),
 \max_{q=0,\ldots,k-1}
 \widehat\alpha^{(k-1-q)}\circ\widehat\gamma(R_r(q))
 \right\}.
\end{equation}
Consequently:
\begin{enumerate}
\item[(i)] If $\rho(t)\leq\bar\rho$ for all $t$, then
\begin{equation}\label{eq:practical_limsup}
 \limsup_{t\to\infty}V(x(t))\leq\widehat\gamma(\bar\rho),
 \qquad
 \limsup_{t\to\infty}\|x(t)\|
 \leq\ul\alpha_V^{-1}(\widehat\gamma(\bar\rho)).
\end{equation}
The finite offset $\Delta_0$ affects only the transient term in
\eqref{eq:intermediate_iteration_with_offset} and does not enlarge this asymptotic bound.

\item[(ii)] If $\rho(t)\to0$, then $V(x(t))\to0$ and $x(t)\to0$.

\item[(iii)] If $\rho(t)=0$ for all $t$, then for every
$t=kM+r$, $r\in\{0,\ldots,M-1\}$,
\begin{equation}\label{eq:full_time_V_bound_offset_case}
 V(x(t))
 \leq
 \widehat\alpha^{(k)}
 \!\left(\kappa_{\mathrm{int}}(V(x(0)))+\Delta_0\right).
\end{equation}
Hence $x(t)\to0$.  However, if $\Delta_0>0$ is a fixed offset independent of
$x(0)$, \eqref{eq:full_time_V_bound_offset_case} by itself does not establish
Lyapunov stability of the origin.

\item[(iv)] Suppose $\rho(t)=0$ for all $t$ and, on the considered domain,
there exists $\delta_\Delta\in\Kinf$ such that
\begin{equation}\label{eq:Delta_vanishing_condition}
 \Delta_0\leq\delta_\Delta(V(x(0))).
\end{equation}
Then the origin is asymptotically stable in $\X$ and admits a $\KL$ estimate.
The exact-consistency case $\Delta_0=0$ is included as a special case.
\end{enumerate}
\end{proposition}

\begin{proof}
Evaluating \eqref{eq:pract_asym_sta_like} at $t=r+kM$ gives
\[
 Y_r(k+1)\leq\max\{\widehat\alpha(Y_r(k)),\widehat\gamma(R_r(k))\}.
\]
Induction yields \eqref{eq:intermediate_iteration_exact}.  Lemma~
\ref{lem:initial_intermediate_K_bound}, evaluated at $\tau=0$, gives
$Y_r(0)=V(x(r))
 \leq\kappa_{\mathrm{int}}(V(x(0)))+\Delta_0$, $r=0,\ldots,M-1$.
Since $\widehat\alpha$ is increasing, substitution of this bound into
\eqref{eq:intermediate_iteration_exact} proves
\eqref{eq:intermediate_iteration_with_offset}.

Because $\widehat\alpha<\id$, every positive iterate of
$\widehat\alpha$ is no larger than its argument and
$\widehat\alpha^{(k)}(s)\to0$ for every fixed $s\geq0$.
If $R_r(k)\leq\bar\rho$, then
\eqref{eq:intermediate_iteration_exact} gives
$\limsup_{k\to\infty}Y_r(k)\leq\widehat\gamma(\bar\rho)$.
There are only $M$ interlaced subsequences
$\{r+kM\}_{k\geq0}$, so \eqref{eq:practical_limsup} follows from
\eqref{eq:V-bounds}.  The first term in
\eqref{eq:intermediate_iteration_with_offset} converges to zero for every finite
$\Delta_0$, which proves the last statement in item~(i).

Now suppose $R_r(k)\to0$.  Given $\varepsilon>0$, choose $N$ such that
$\widehat\gamma(R_r(k))\leq\varepsilon$ for $k\geq N$.
Iteration from $N$ yields
$ Y_r(N+m)\leq\max\{\widehat\alpha^{(m)}(Y_r(N)),\varepsilon\}$.
The first term converges to zero, hence $Y_r(k)\to0$.  

Repeating the argument for
$r=0,\ldots,M-1$ proves (ii).

If $\rho\equiv0$, the second term in
\eqref{eq:intermediate_iteration_with_offset} vanishes and
\eqref{eq:full_time_V_bound_offset_case} follows immediately.  Since the argument of
$\widehat\alpha^{(k)}$ is finite and
$\widehat\alpha^{(k)}(s)\to0$, item~(iii) follows.
A fixed positive $\Delta_0$, however, prevents the right-hand side at the first $M-1$ time instants
from tending to zero with $V(x(0))$; therefore this estimate alone is insufficient
for Lyapunov stability.

For (iv), define $\widetilde\kappa_{\mathrm{int}}
 :=\kappa_{\mathrm{int}}+\delta_\Delta\in\Kinf$.
By \eqref{eq:Delta_vanishing_condition},
\eqref{eq:full_time_V_bound_offset_case} yields
 $V(x(t))
 \leq
 \widehat\alpha^{(k)}
 \circ\widetilde\kappa_{\mathrm{int}}(V(x(0)))$ for $t=kM+r$.
For integer $k\geq0$, set
$\zeta(s,k):=\widehat\alpha^{(k)}
(\widetilde\kappa_{\mathrm{int}}(s))$ and interpolate linearly with respect to
$k$ on every interval $[k,k+1]$.  Since
$\widehat\alpha\in\Kinf$ and $\widehat\alpha<\id$, the resulting function
$\bar\beta_V\in\KL$.  With
$k=\lfloor t/M\rfloor\geq\max\{t/M-1,0\}$ and monotonicity of
$\bar\beta_V$ in its second argument,
$
 V(x(t))
 \leq
 \bar\beta_V\!\left(
 V(x(0)),\max\{t/M-1,0\}
 \right).
$
Using both inequalities in \eqref{eq:V-bounds} gives
$
 \|x(t)\|
 \leq
 \ul\alpha_V^{-1}\!\left(
 \bar\beta_V\!\left(
 \ol\alpha_V(\|x(0)\|),\max\{t/M-1,0\}
 \right)\right),
$
which is a $\KL$ estimate.  This proves asymptotic stability at all physical
time instants.
\end{proof}

\subsection{Radius and mismatch bounds}

\begin{lemma}[Minimal optimized radii and auxiliary terminal radius]
\label{lem:min_tube_equality_rigorous}
For fixed packet data, suppose an optimizer of \eqref{eq:OCP2} exists.  Then an optimal solution
can be selected such that
\begin{align}
\label{eq:tube_eq_rigorous}
 r_i^\star(0|t)&=0,\nonumber\\
 r_i^\star(l+1|t)&=L_{i,1}r_i^\star(l|t)+\mathcal B_i(l|t),
 \quad l=0,\ldots,M-2,
\end{align}
and $r_i^\star(M|t)$ is given by \eqref{eq:terminal_radius_aux}.
\end{lemma}

\begin{proof}
Starting from any feasible radius vector, recursively replace every optimized radius by the
smallest value allowed by its propagation inequality.  Because $L_{i,1}\geq0$, the resulting
sequence is componentwise no larger.  Decreasing a radius enlarges
$\X_i\ominus\B_{r_i(l|t)}$, so all state constraints remain feasible, and all radius inequalities
remain satisfied.  If any optimized inequality is strict, at least one penalized radius decreases
strictly, reducing the objective because $\lambda>0$.  Thus an optimal tight-radius selection
exists.  The terminal value is fixed by definition and is not optimized.
\end{proof}

\begin{lemma}[Explicit tightening-radius and mismatch bounds]
\label{lem:ri_bound_monotone_rigorous}
For a tight-radius optimizer and $l=1,\ldots,M$,
\begin{equation}\label{eq:ri_explicit_rigorous}
 r_i^\star(l|t)
 =\sum_{k=0}^{l-1}L_{i,1}^{l-1-k}\mathcal B_i(k|t).
\end{equation}
Moreover,
\begin{equation}\label{eq:s_explicit_rigorous}
 s_i(l|t)
 =\sum_{k=0}^{l-1}L_{i,1}^{l-1-k}\eta_i(k|t).
\end{equation}
Consequently,
\begin{align}
\label{eq:ri_uniform_bound_rigorous}
 \max_{l=0,\ldots,M}r_i^\star(l|t)
 &\leq\left(\sum_{p=0}^{M-1}L_{i,1}^{p}\right)
       \max_{k=0,\ldots,M-1}\mathcal B_i(k|t),\\
\label{eq:si_uniform_bound_rigorous}
 \max_{l=0,\ldots,M}s_i(l|t)
 &\leq\left(\sum_{p=0}^{M-1}L_{i,1}^{p}\right)
       \max_{k=0,\ldots,M-1}\eta_i(k|t).
\end{align}
Both recursions are order preserving in their nonnegative driving signals.
\end{lemma}

\begin{proof}
Equations \eqref{eq:ri_explicit_rigorous} and \eqref{eq:s_explicit_rigorous} follow by induction
from \eqref{eq:tube_eq_rigorous},~\eqref{eq:terminal_radius_aux}, and
\eqref{eq:s_correction_recursion}.
Bounding every driving term by its horizon maximum gives
\eqref{eq:ri_uniform_bound_rigorous}--\eqref{eq:si_uniform_bound_rigorous}.
Nonnegativity of all powers of $L_{i,1}$ proves monotonicity.
\end{proof}

\begin{corollary}[Vanishing terminal practical term]
\label{cor:ri_vanish_rigorous}
If, for every $i$, $\max_{k=0,\ldots,M-1}\mathcal B_i(k|t)\to 0$,
$\max_{k=0,\ldots,M-1}\eta_i(k|t)\to 0$, then $\rho(t)\to0$ and hence $x(t)\to0$ by
Proposition~\ref{prop:practical_form_rigorous}(ii).
\end{corollary}

\section{Constrained linear networks}\label{sec:Constrained linear networks}

This section gives finite-dimensional certificates, independent of
terminal-set construction, for linear systems and local quadratic finite-step Lyapunov-like functions.  They concern the finite-step feedback control laws used to establish Assumption~\ref{ass:clf-small-gain} and,
when constraints are present, a certified operating domain on which that feedback is admissible.
%\subsection{Linear network, quadratic local functions, and finite-step feedback control laws}
Consider
\begin{equation}\label{eq:lin_network}
 x(t+1)=Ax(t)+Bu(t),\qquad
 x=(x_1,\ldots,x_\ell),\quad u=(u_1,\ldots,u_\ell),
\end{equation}
where $A=[A_{ij}]_{i,j=1}^\ell$ and
$B=\operatorname{blkdiag}(B_1,\ldots,B_\ell)$.  Let
$K=\operatorname{blkdiag}(K_1,\ldots,K_\ell)$ and define
$A_K:=A+BK$.
For an initial condition $\xi$, the associated length-$M$ finite-step feedback sequence is
\begin{equation}\label{eq:q_linear_sequence}
 q(\xi):=\bigl(K\xi,\,KA_K\xi,\,\ldots,\,KA_K^{M-1}\xi\bigr).
\end{equation}
Thus the corresponding state trajectory is $x(l,\xi,q)=A_K^l\xi$, $l=0,\ldots,M$.
Equation~\eqref{eq:q_linear_sequence}, rather than the one-step expression $K\xi$ alone, is the map
$q:\X\to\U^M$ appearing in Assumption~\ref{ass:clf-small-gain}.

Fix $P_i=P_i^\top\succ0$ and define
$W_i(\xi_i):=\xi_i^\top P_i\xi_i$.
Then Assumption~\ref{ass:clf-small-gain}(i) holds with
$\lambda_{\min}(P_i)\|\xi_i\|^2\leq W_i(\xi_i)
 \leq\lambda_{\max}(P_i)\|\xi_i\|^2$.
Write
\begin{equation}\label{eq:Phi_blocks}
 \Phi(M):=A_K^M=[\Phi_{ij}(M)]_{i,j=1}^\ell,
 \qquad
 c_{ij}(M):=\left\|P_i^{1/2}\Phi_{ij}(M)P_j^{-1/2}\right\|_2.
\end{equation}
Here $P_i^{1/2}$ denotes the unique symmetric positive-definite principal square root of $P_i$, and $P_i^{-1/2}:=(P_i^{1/2})^{-1}$ denotes its inverse.  Equivalently, $P_i^{-1/2}$ is the principal square root of $P_i^{-1}$.

\begin{definition}\label{def:cij}
For each $i$, let
$\mathcal J_i:=\{j:c_{ij}(M)>0\}$.  Choose $a_{ij}>0$ for
$j\in\mathcal J_i$ such that
\begin{equation}\label{eq:gain_allocation_condition}
 \sum_{j\in\mathcal J_i}\frac{c_{ij}(M)}{\sqrt{a_{ij}}}\leq1.
\end{equation}
Set $a_{ij}:=0$ when $c_{ij}(M)=0$ and define
\begin{equation}\label{eq:linear_gains}
 \gamma_{ij}(s):=a_{ij}s.
\end{equation}
A universally available choice is
\begin{equation}\label{eq:uniform_gain_choice}
 a_{ij}=d_i^2c_{ij}^2(M),\qquad d_i:=|\mathcal J_i|, j\in\mathcal J_i.
\end{equation}
\end{definition}

\begin{theorem}
\label{thm:ass1ii_linear}
Let \eqref{eq:lin_network}--\eqref{eq:Phi_blocks} hold and let the gains satisfy
\eqref{eq:gain_allocation_condition}.  Then, for every $\xi\in\R^n$,
\begin{equation}\label{eq:linear_max_gain_bound}
 W_i\bigl(x_i(M,\xi,q)\bigr)
 \leq\max_{j=1,\ldots,\ell}\gamma_{ij}\bigl(W_j(\xi_j)\bigr),
 \qquad i=1,\ldots,\ell.
\end{equation}
Hence Assumption~\ref{ass:clf-small-gain}(ii) holds for the finite-step feedback control laws
\eqref{eq:q_linear_sequence}.
\end{theorem}

\begin{proof}
Set $z_j=P_j^{1/2}\xi_j$.  Since
$x_i(M,\xi,q)=\sum_j\Phi_{ij}(M)\xi_j$,
\begin{align*}
 \sqrt{W_i(x_i(M,\xi,q))}
 &\leq\sum_{j\in\mathcal J_i}c_{ij}(M)\|z_j\|_2\\
 &=\sum_{j\in\mathcal J_i}
   \frac{c_{ij}(M)}{\sqrt{a_{ij}}}
   \sqrt{a_{ij}W_j(\xi_j)}\\
 &\leq
 \left(\sum_{j\in\mathcal J_i}\frac{c_{ij}(M)}{\sqrt{a_{ij}}}\right)
 \max_{j\in\mathcal J_i}\sqrt{a_{ij}W_j(\xi_j)}.
\end{align*}
Condition~\eqref{eq:gain_allocation_condition} and squaring give
\eqref{eq:linear_max_gain_bound}.  For \eqref{eq:uniform_gain_choice}, every nonzero summand in
\eqref{eq:gain_allocation_condition} equals $1/d_i$, so the sum equals one.
\end{proof}

%The norm allocation above is explicit.  
The following semidefinite certificate can be less conservative
for a prescribed family of gains.

\begin{proposition}[SDP certificate for prescribed quadratic gains]
\label{prop:gain_sdp_certificate}
Let $\Phi_i(M):=[\Phi_{i1}(M)\ \cdots\ \Phi_{i\ell}(M)]$
be the $i$th block row of $A_K^M$.  Fix $a_{ij}>0$ for every block for which
$\Phi_{ij}(M)\neq0$.  If there exist $\tau_{ij}\geq0$ such that
\begin{align}
\label{eq:gain_sdp_matrix}
 \Phi_i(M)^\top P_i\Phi_i(M)
 &\preceq\operatorname{blkdiag}(\tau_{i1}P_1,\ldots,\tau_{i\ell}P_\ell),\\
\label{eq:gain_sdp_scalar}
 \sum_{j=1}^\ell\frac{\tau_{ij}}{a_{ij}}&\leq1,
\end{align}
then \eqref{eq:linear_max_gain_bound} holds for subsystem $i$.
For fixed $P_i$, $K$, $M$, and $a_{ij}$, conditions
\eqref{eq:gain_sdp_matrix}--\eqref{eq:gain_sdp_scalar} form an SDP feasibility test.
\end{proposition}

\begin{proof}
Let $s:=\max_j a_{ij}W_j(\xi_j)$.  Then
$W_j(\xi_j)\leq s/a_{ij}$.  By \eqref{eq:gain_sdp_matrix}, we get 
$ W_i(x_i(M,\xi,q))
 \leq\sum_j\tau_{ij}W_j(\xi_j)
 \leq s\sum_j\frac{\tau_{ij}}{a_{ij}}\leq s$.
\end{proof}

\subsection{Tractable verification of the cyclic small-gain condition}

\begin{lemma}
%[Equivalent cycle, scaling, and linear-program certificates]
\label{lem:cycle_linear}
Let $A_\gamma:=[a_{ij}]_{i,j=1}^\ell$ be the nonnegative gain matrix associated with
\eqref{eq:linear_gains}. The following statements are equivalent.
\begin{enumerate}
\item Every directed simple cycle $(i_1,\ldots,i_r)$ satisfies
\begin{equation}\label{eq:prod_cycle}
 a_{i_1i_2}a_{i_2i_3}\cdots a_{i_ri_1}<1.
\end{equation}
\item There exist $p_i>0$ and $\theta\in(0,1)$ such that
\begin{equation}\label{eq:scaled_edge_condition}
 a_{ij}p_j\leq\theta p_i
 \qquad\text{for every edge }(j,i)\text{ with }a_{ij}>0.
\end{equation}
\item With variables $y_i=\log p_i$ and $\vartheta=\log\theta<0$, the difference
constraints
\begin{equation}\label{eq:log_scaling_lp}
 \log a_{ij}+y_j-y_i\leq\vartheta
\end{equation}
are feasible for all nonzero edges.
\end{enumerate}
When these conditions hold, the explicit scaling functions
$\sigma_i(s):=p_i s$ satisfy
%\begin{equation}\label{eq:linear_scaled_gain}
$\sigma_i^{-1}\circ\gamma_{ij}\circ\sigma_j(s)
 \leq\theta s$,
%\end{equation}
and therefore
\begin{equation}\label{eq:linear_global_V}
 V(x):=\max_i\frac{x_i^\top P_i x_i}{p_i}
\end{equation}
obeys $V(A_K^M x)\leq\theta V(x)$ under the feedback
\eqref{eq:q_linear_sequence}.
\end{lemma}

\begin{proof}
For linear gains, Assumption~\ref{ass:clf-small-gain}(iii) is equivalent to \eqref{eq:prod_cycle}.  Taking logarithms converts cycle products
into cycle sums.  The strict negativity of every cycle sum is equivalent to feasibility of the
strict difference constraints \eqref{eq:log_scaling_lp}; a common negative margin can be chosen
because the graph is finite.  Exponentiation gives \eqref{eq:scaled_edge_condition}, and the
reverse implication follows by multiplying \eqref{eq:scaled_edge_condition} around any cycle.
Finally,
$\sigma_i^{-1}\circ\gamma_{ij}\circ\sigma_j(s)
=(a_{ij}p_j/p_i)s\leq\theta s$.  Taking maxima and using
Theorem~\ref{thm:ass1ii_linear} yields the last claim.
\end{proof}

\begin{remark}
%[Computational implementation]
For fixed gains $a_{ij}$, minimizing $\vartheta$ subject to
\eqref{eq:log_scaling_lp} and one normalization such as $y_1=0$ is a linear program.  Its
optimal value is the logarithm of the maximum cycle geometric mean.  Thus the complete cyclic
small-gain condition can be checked without enumerating cycles.  The same LP returns the weights
$p_i$ used in the explicit global function \eqref{eq:linear_global_V}.
\end{remark}

\subsection{Constraint admissibility on a certified operating domain}
Assumption~\ref{ass:clf-small-gain} requires the complete length-$M$ feedback sequence to be admissible.  For a
constrained linear plant this property must be checked in addition to the gain inequalities.
Let
$ \X=\{x:H_xx\leq h_x\}$, $\U=\{u:H_uu\leq h_u\}$, where every component of $h_x,h_u$ is positive, and consider the ellipsoid $\Omega(Q):=\{x:x^\top Q^{-1}x\leq1\}$, $Q\succ0$.

\begin{proposition}[Exact ellipsoidal admissibility certificate]
\label{prop:linear_admissibility_domain}
For fixed $K$ and $M$, suppose
\begin{align}
\label{eq:M_invariance_Q}
 A_K^M Q(A_K^M)^\top&\preceq Q,\\
\label{eq:state_support_Q}
 H_{x,k}A_K^rQ(A_K^r)^\top H_{x,k}^\top&\leq h_{x,k}^2,
 &&r=0,\ldots,M,\\
\label{eq:input_support_Q}
 H_{u,k}KA_K^rQ(A_K^r)^\top K^\top H_{u,k}^\top&\leq h_{u,k}^2,
 &&r=0,\ldots,M-1,
\end{align}
for every row $H_{x,k}$ and $H_{u,k}$.  Then $q$ in
\eqref{eq:q_linear_sequence} is an admissible finite-step feedback control laws on $\Omega(Q)$, and the
block-boundary state $A_K^M x$ remains in $\Omega(Q)$.  Moreover,
\begin{equation}\label{eq:linear_K_bounded_intermediate}
 \|A_K^r x\|\leq\|A_K^r\|_2\|x\|,
 \qquad r=0,\ldots,M,
\end{equation}
verifies the finite-step $\K$-boundedness requirement.
For fixed $K$ and $M$, \eqref{eq:M_invariance_Q}--\eqref{eq:input_support_Q} are LMIs or
scalar linear inequalities in $Q$.
\end{proposition}

\begin{proof}
For an ellipsoid $\Omega(Q)$,
$\max_{x\in\Omega(Q)}c^\top x=\sqrt{c^\top Qc}$.  Applying this support formula to
$c^\top=H_{x,k}A_K^r$ and $c^\top=H_{u,k}KA_K^r$ proves state and input admissibility.
Condition~\eqref{eq:M_invariance_Q} is exactly the ellipsoidal containment
$A_K^M\Omega(Q)\subseteq\Omega(Q)$.  Equation~\eqref{eq:linear_K_bounded_intermediate} is
immediate from linearity.
\end{proof}

\begin{remark}
%[Role of $\Omega(Q)$]
The set $\Omega(Q)$ is an offline certified operating domain for the finite-step feedback used to
verify Assumption~\ref{ass:clf-small-gain}.  It is not imposed as a terminal set in Problem~\ref{prob:OCP-2}.  If
Assumption~\ref{ass:clf-small-gain} is claimed on the full constrained set $\X$, the corresponding support conditions
must be verified on $\X$ itself; Proposition~\ref{prop:linear_admissibility_domain} gives a
tractable inner-domain certificate when a global claim is unnecessary or false.
\end{remark}

%\subsection{Constructive design workflow}
A coherent linear design can be carried out as follows.
\begin{enumerate}
\item Choose a block-diagonal $K$ and a candidate finite-step length $M$.
\item Choose $P_i\succ0$ and compute the blocks of $A_K^M$.
\item Construct gains by \eqref{eq:gain_allocation_condition}, or certify prescribed gains with
Proposition~\ref{prop:gain_sdp_certificate}.
\item Solve the LP \eqref{eq:log_scaling_lp}.  If its optimal $\theta$ is not strictly below one,
increase $M$ or redesign $K$ or $P_i$.
\item If constraints are part of the Assumption~\ref{ass:clf-small-gain} domain, solve
\eqref{eq:M_invariance_Q}--\eqref{eq:input_support_Q} for a nontrivial $Q$.
\end{enumerate}
The gain construction and the cyclic small-gain test concern finite-step decay; the ellipsoidal
support tests concern admissibility.  Neither one substitutes for the other.
For example consider
\begin{equation}\label{eq:linear_two_example}
 x^+=Ax,\qquad
 A=\begin{bmatrix}1.2&-10\\0.05&-1.2\end{bmatrix}.
\end{equation}
Direct multiplication gives $A^2=0.94I_2$.  With $M=2$, $P_1=P_2=1$, and
$W_i(x_i)=x_i^2$, one has $c_{11}=c_{22}=0.94$ and $c_{12}=c_{21}=0$.  The choice
$a_{11}=a_{22}=0.94^2=0.8836$ therefore satisfies
\eqref{eq:gain_allocation_condition}, and every nonzero cycle gain is strictly below one.  At
$M=1$, by contrast, the self-gains induced by the diagonal entries have magnitude larger than
one.  This example illustrates that the certificate genuinely exploits finite-step cancellation.

% -----------------------------------------------------------------------------
%\section{Constrained linear networks: tractable shift, completion, and regularity certificates}
%\label{sec:linear_shift_completion}
% -----------------------------------------------------------------------------

We now specialize the certificates to the
linear-quadratic case, without introducing a terminal set into the OCP.  Write the local dynamics
as
%\begin{equation}\label{eq:local_linear_dynamics}
$x_i^+=A_{ii}x_i+\sum_{j\in\Ni}A_{ij}x_j+B_i u_i
 =A_{ii}x_i+A_{i\Ni}x_{\Ni}+B_i u_i$,
%\end{equation}
and let
\begin{equation}\label{eq:local_poly_constraints}
 \X_i=\{x_i:F_i x_i\leq f_i\},\qquad
 \U_i=\{u_i:G_i u_i\leq g_i\}.
\end{equation}
The local finite-step functions remain $W_i(x_i)=x_i^\top P_i x_i$.
%\subsection{Explicit incremental constants}
The Lipschitz constants as in~\eqref{eq:Lipschitz_dynamics_p} can be computed as
%\begin{equation}\label{eq:linear_L_constants}
$L_{i,1}=\|A_{ii}\|_2,
 L_{i,2}=\sum_{j\in\Ni}\|A_{ij}\|_2,
 L_{i,3}=\|B_i\|_2$.
%\end{equation}
Tighter constants may be obtained by
computing the induced norm of the block row $A_{i\Ni}$ from
$(\prod_{j\in\Ni}\R^{n_j},\|\cdot\|_{\Ni,\infty})$ to $(\R^{n_i},\|\cdot\|_2)$.

\subsection{Exact verification of the shift-compatible margins}
For the polyhedral set \eqref{eq:local_poly_constraints}, Lemma \ref{lem:shift_margin_verification} reduces condition
\eqref{eq:shift_margin_condition} to the finite family
\begin{equation}\label{eq:linear_shift_row_check}
 F_{i,k}x_i^{\nom,\star}(l+1|t)
 +\|F_{i,k}\|_2\bigl(\chi_i(l|t+1)+\tilde r_i(l|t+1)\bigr)
 \leq f_{i,k}
\end{equation}
for every row $k$ and every $l=1,\ldots,M-1$.  All quantities in
\eqref{eq:linear_shift_row_check} are known before the OCP at time $t+1$ is accepted.  Thus
condition~\eqref{eq:shift_margin_condition} is verified by matrix-vector products and scalar
comparisons only.
If an offline envelope is desired, suppose
$x_i^{\nom,\star}(l+1|t)$ lies in
$\mathcal E(c_{i,l},Q_{i,l})=
\{c_{i,l}+Q_{i,l}^{1/2}w:\|w\|_2\leq1\}$ and
$\chi_i(l|t+1)+\tilde r_i(l|t+1)\leq\bar\delta_{i,l}$.  Then the uniform row conditions
%\begin{equation}\label{eq:linear_shift_offline_envelope}
 $F_{i,k}c_{i,l}
 +\sqrt{F_{i,k}Q_{i,l}F_{i,k}^\top}
 +\|F_{i,k}\|_2\bar\delta_{i,l}
 \leq f_{i,k}$
%\end{equation}
certify \eqref{eq:linear_shift_row_check} for every point in the envelope.  These checks are exact
for the stated ellipsoidal envelope and can be imposed as second-order-cone inequalities when an
ellipsoid factor is used as the design variable.

\subsection{Exact one-step terminal feasibility tests}
At time $t+1$, set
 $a_i:=A_{ii}\tilde x_i^{\nom}(M-1|t+1)$, $
 +A_{i\Ni}\bar x_{\Ni}(M-1|t+1)$, $b_i:=b_i(t+1)$.
Problem~\ref{prob:one_step_completion_problem} is equivalent to the second-order-cone feasibility problem
\begin{equation}\label{eq:linear_completion_socp_full}
 \text{find }v_i\quad\text{s.t.}\quad
 G_i v_i\leq g_i,
 \qquad
 \|P_i^{1/2}(a_i+B_i v_i)\|_2\leq\sqrt{b_i}.
\end{equation}
Thus any feasible $v_i$ is an admissible final input for the shifted candidate.  Equivalently, define
\begin{equation}\label{eq:linear_completion_value_qp}
 \Psi_i(a_i):=\min_{G_i v_i\leq g_i}(a_i+B_i v_i)^\top P_i(a_i+B_i v_i).
\end{equation}
Then the one-step terminal condition is feasible if and only if
$\Psi_i(a_i)\leq b_i$.  Problem~\eqref{eq:linear_completion_value_qp} is a convex quadratic
program and \eqref{eq:linear_completion_socp_full} is its SOCP feasibility representation.

\begin{lemma}[Closed-form unconstrained terminal value]
\label{lem:unconstrained_completion}
If $\U_i=\R^{m_i}$, define
$R_i:=B_i^\top P_iB_i$ and let $R_i^\dagger$ denote the Moore--Penrose inverse.  A
minimum-norm minimizer is
$v_i^\star=-R_i^\dagger B_i^\top P_i a_i$, and
\begin{align}
\label{eq:unconstrained_completion_value}
 \Psi_i^{\rm unc}(a_i)
 &=a_i^\top\left(P_i-P_iB_iR_i^\dagger B_i^\top P_i\right)a_i.
\end{align}
Consequently, $\Psi_i^{\rm unc}(a_i)\leq b_i$ is an exact analytical test in the
unconstrained-input case and a necessary lower-bound test when input constraints are present.
\end{lemma}

\begin{proof}
Equation~\eqref{eq:linear_completion_value_qp} without input constraints is a weighted
least-squares problem.  Its normal equation is
$R_i v_i=-B_i^\top P_i a_i$.  The pseudoinverse gives the minimum-norm solution, and direct
substitution gives \eqref{eq:unconstrained_completion_value}.
\end{proof}

\begin{remark}
The SOCP \eqref{eq:linear_completion_socp_full} is an online test for the shifted
prefix available at time $t+1$.  A uniform offline guarantee over a prescribed family of shifted
prefixes requires an outer approximation of the reachable triples
$(\tilde x_i^{\nom}(M-1|t+1),\bar x_{\Ni}(M-1|t+1),b_i(t+1))$.  Given a polytopic or
ellipsoidal outer approximation, robust LP/SOCP/SDP conditions with an affine final-input law
provide sufficient offline tests.  Without such a reachable-set description, Assumption~\ref{ass:clf-small-gain} alone
does not imply feasibility of the final input for an arbitrary shifted prefix.
\end{remark}

\subsection{Quadratic prediction-mismatch and intermediate-time bounds}
Let $p_i$ and $\theta$ satisfy \eqref{eq:scaled_edge_condition}, and let $V$ be given by
\eqref{eq:linear_global_V}.  Define
%\begin{equation}\label{eq:linear_cP_cV}
$c_P:=\max_i\frac{\lambda_{\max}(P_i)}{p_i},
 c_V:=\frac{1}{\ell}\min_i\frac{\lambda_{\min}(P_i)}{p_i}>0$.
%\end{equation}
Then
\begin{equation}\label{eq:linear_V_quadratic_bounds}
 c_V\|x\|_2^2\leq V(x)\leq c_P\|x\|_2^2.
\end{equation}
Moreover, for every $x,z\in\R^n$ and every $\nu>0$,
\begin{equation}\label{eq:linear_quadratic_perturbation}
 V(x+z)
 \leq (1+\nu)V(x)
 +(1+\nu^{-1})c_P\max_i\|z_i\|_2^2.
\end{equation}
Indeed, for each subsystem,
$\|P_i^{1/2}(x_i+z_i)\|_2^2
\leq(1+\nu)x_i^\top P_i x_i+(1+\nu^{-1})z_i^\top P_i z_i$; division by
$p_i$ and maximization over $i$ give \eqref{eq:linear_quadratic_perturbation}.  This global
quadratic estimate is sharper for the present linear-quadratic setting than invoking a Lipschitz
modulus on a compact set.

To specialize Lemma~\ref{lem:initial_intermediate_K_bound}, let $E_i$ denote the block selector
satisfying $x_i=E_i x$, and define
%\begin{equation}\label{eq:linear_q_cost_coeff}
$c_{q,i}:=\sum_{l=0}^{M-1}\|P_i^{1/2}E_iA_K^l\|_2^2,
 c_q:=\max_i\frac{c_{q,i}}{p_i}$.
%\end{equation}
For the optimal value $J_i^\star(t)$ of the local OCP, set
%\begin{align}\label{eq:linear_deltaJ}
$\delta_{J,i}^L(t)
 :=\max\{J_i^\star(t)-c_{q,i}\|x(t)\|_2^2,0\},
 \delta_J^L(t):=\max_i\frac{\delta_{J,i}^L(t)}{p_i},$
%\end{align}
and define the largest actual-to-nominal mismatch bound over the intermediate indices by
%\begin{equation}\label{eq:linear_epsilon_int}
$\varepsilon_L(t):=
 \max_{r=0,\ldots,M-1}\max_i
 \bigl(r_i^\star(r|t)+s_i(r|t)\bigr)$.
%\end{equation}
Fix any $\nu_{\rm int}>0$ and put
\begin{equation}\label{eq:linear_kappa_Delta}
 \kappa_L:= (1+\nu_{\rm int})\frac{c_q}{c_V},
 \qquad
 \Delta_L(t):=(1+\nu_{\rm int})\delta_J^L(t)
 +(1+\nu_{\rm int}^{-1})c_P\varepsilon_L(t)^2.
\end{equation}

\begin{lemma}
%[Intermediate-time estimate for the constrained linear network]
\label{lem:linear_intermediate_bound}
Suppose the local OCPs are recursively feasible.  Then, for every $t\in\Zp$ and
$r=0,\ldots,M-1$,
\begin{equation}\label{eq:linear_intermediate_bound}
 V(x(t+r))\leq \kappa_L V(x(t))+\Delta_L(t).
\end{equation}
In addition, $x_i(t+r)\in\X_i$ for every subsystem and every such intermediate physical time.
\end{lemma}

\begin{proof}
The constraint statement follows from Corollary~\ref{cor:true_constraint_satisfaction}.  We prove
\eqref{eq:linear_intermediate_bound}.  Along the admissible linear feedback
\eqref{eq:q_linear_sequence},
 $W_i(E_iA_K^l x)\leq
 \|P_i^{1/2}E_iA_K^l\|_2^2\|x\|_2^2$,
so $c_{q,i}\|x(t)\|_2^2$ is the explicit linear-quadratic counterpart of the comparison
quantity in Lemma~\ref{lem:initial_intermediate_K_bound}.  Since every term in the local OCP
objective is nonnegative, then
 $W_i(x_i^{\nom,\star}(r|t))\leq J_i^\star(t)
 \leq c_{q,i}\|x(t)\|_2^2+\delta_{J,i}^L(t)$.
After division by $p_i$ and maximization over $i$,
\begin{equation}\label{eq:linear_nom_intermediate}
 V(x^{\nom,\star}(r|t))
 \leq c_q\|x(t)\|_2^2+\delta_J^L(t)
 \leq \frac{c_q}{c_V}V(x(t))+\delta_J^L(t),
\end{equation}
where the last inequality uses \eqref{eq:linear_V_quadratic_bounds}.
Lemma~\ref{lem:corrected_actual_nominal_bound} gives
$\|x_i(t+r)-x_i^{\nom,\star}(r|t)\|_2
\leq r_i^\star(r|t)+s_i(r|t)\leq\varepsilon_L(t)$.
Applying \eqref{eq:linear_quadratic_perturbation} with $\nu=\nu_{\rm int}$ to
\eqref{eq:linear_nom_intermediate} yields
\eqref{eq:linear_intermediate_bound}--\eqref{eq:linear_kappa_Delta}.
\end{proof}

The next theorem collects the linear certificates into one terminal-set-free result.  It separates
offline finite-step and admissibility tests from the online shift-margin and one-step terminal
feasibility tests and from the measured prediction/reoptimization mismatch.

\begin{theorem}
%[Constrained linear DMPC: recursive feasibility and practical stability]
\label{thm:cln_RF_stability}
Consider \eqref{eq:lin_network} with local polyhedral constraints
\eqref{eq:local_poly_constraints}, Problem~\ref{prob:OCP-2}, and
Algorithm~\ref{alg:dist_multi_step_mpc}.  Suppose:
\begin{enumerate}
\item[(L1)] $P_i\succ0$, a block-diagonal $K$, a finite-step length $M$, and gains $a_{ij}$
satisfy either Theorem~\ref{thm:ass1ii_linear} or
Proposition~\ref{prop:gain_sdp_certificate}, and the LP
\eqref{eq:log_scaling_lp} returns $p_i>0$ and $\theta\in(0,1)$;
\item[(L2)] the finite-step feedback \eqref{eq:q_linear_sequence} is admissible on the
Assumption~\ref{ass:clf-small-gain} domain; Proposition~\ref{prop:linear_admissibility_domain} gives one sufficient
ellipsoidal test;
\item[(L3)] the local OCPs are feasible at $t=0$, and at every transition the exact row tests
\eqref{eq:linear_shift_row_check} and the one-step terminal-feasibility SOCP
\eqref{eq:linear_completion_socp_full} are feasible.
\end{enumerate}
Define
\begin{equation}\label{eq:linear_rho_definition_new}
 \rho(t):=\max_i\bigl(r_i^\star(M|t)+s_i(M|t)\bigr).
\end{equation}
Choose $\nu_M>0$ so that
%\begin{equation}\label{eq:linear_a_b}
$a_L:=(1+\nu_M)\theta<1,
 b_L:=(1+\nu_M^{-1})c_P$.
%\end{equation}
Then:
\begin{enumerate}
\item[(a)] the OCPs are recursively feasible and
$x(t)\in\X$, $u(t)\in\U$ for all $t\in\Zp$;
\item[(b)] the intermediate physical times satisfy
\begin{equation}\label{eq:linear_intermediate_theorem}
 V(x(t+r))\leq\kappa_LV(x(t))+\Delta_L(t),
 \qquad r=0,\ldots,M-1;
\end{equation}
\item[(c)] the closed-loop satisfies the quadratic $M$-step estimate
\begin{equation}\label{eq:linear_additive_Mstep}
 V(x(t+M))\leq a_LV(x(t))+b_L\rho(t)^2;
\end{equation}
\item[(d)] for every $r\in\{0,\ldots,M-1\}$ and $k\in\Zp$,
%\begin{align}\label{eq:linear_interlaced_bound}
 $V(x(r+kM))
 \leq a_L^k\bigl(\kappa_LV(x(0))+\Delta_L(0)\bigr)+b_L\sum_{q=0}^{k-1}a_L^{k-1-q}\rho(r+qM)^2$,
%\end{align}
where the sum is zero for $k=0$;
\item[(e)] if $\rho(t)\leq\bar\rho$ for all $t$, then
%\begin{equation}\label{eq:linear_practical_limsup}
$ \limsup_{t\to\infty}V(x(t))
 \leq\frac{b_L}{1-a_L}\bar\rho^2,
 \limsup_{t\to\infty}\|x(t)\|_2
 \leq
 \sqrt{\frac{b_L}{c_V(1-a_L)}}\,\bar\rho$,
%\end{equation}
\item[(f)] if $\rho(t)\to0$, then $x(t)\to0$;
\item[(g)] if $\rho(t)=0$ for all $t$ and, on the considered domain, there is a constant
$d_\Delta\geq0$ such that
\begin{equation}\label{eq:linear_Delta_state_bound}
 \Delta_L(0)\leq d_\Delta V(x(0)),
\end{equation}
then the origin is asymptotically stable on that domain.  More precisely,
\begin{equation}\label{eq:linear_fulltime_AS_bound}
 V(x(r+kM))\leq
 a_L^k(\kappa_L+d_\Delta)V(x(0)),
 \qquad r=0,\ldots,M-1,
\end{equation}
which yields a $\KL$ estimate in the stacked Euclidean norm through
\eqref{eq:linear_V_quadratic_bounds}.
\end{enumerate}
\end{theorem}

\begin{proof}
Part (a) follows from Theorem~\ref{thm:online_tube_rec_feas}: the row tests
\eqref{eq:linear_shift_row_check} are equivalent to Assumption~\ref{ass:shift_margin_condition}
and \eqref{eq:linear_completion_socp_full} is exactly Problem~\ref{prob:one_step_completion_problem}
for the linear-quadratic subsystem.  Corollary~\ref{cor:true_constraint_satisfaction} gives the
state and input constraints.  Part (b) is Lemma~\ref{lem:linear_intermediate_bound}.
By Lemma~\ref{lem:cycle_linear}, the terminal inequalities of the OCP imply
$V(x^{\nom,\star}(M|t))\leq\theta V(x(t))$.  Lemma~\ref{lem:corrected_actual_nominal_bound}
and \eqref{eq:linear_rho_definition_new} give
$\max_i\|x_i(t+M)-x_i^{\nom,\star}(M|t)\|_2\leq\rho(t)$.
Applying \eqref{eq:linear_quadratic_perturbation} with $\nu=\nu_M$ proves
\eqref{eq:linear_additive_Mstep}.

For a fixed intermediate index $r$, define
$Y_r(k):=V(x(r+kM))$ and $R_r(k):=\rho(r+kM)$.  Equation
\eqref{eq:linear_additive_Mstep} gives
$Y_r(k+1)\leq a_LY_r(k)+b_LR_r(k)^2$.  Iteration yields
 $Y_r(k)\leq a_L^kY_r(0)
 +b_L\sum_{q=0}^{k-1}a_L^{k-1-q}R_r(q)^2$.

Using \eqref{eq:linear_intermediate_theorem} at $t=0$ gives
$Y_r(0)\leq\kappa_LV(x(0))+\Delta_L(0)$, proving (d).
If $R_r(k)\leq\bar\rho$, the geometric series gives
$\limsup_{k\to\infty}Y_r(k)\leq b_L\bar\rho^2/(1-a_L)$.  Since there are only $M$
interlaced sequences and $c_V\|x\|_2^2\leq V(x)$, (e) follows.

If $R_r(k)\to0$, the convolution with the exponentially decaying sequence $a_L^k$ tends to
zero: for any $\epsilon>0$, split the sum at a sufficiently large index $N$ so that
$b_LR_r(q)^2\leq\epsilon(1-a_L)/2$ for $q\geq N$; the finite prefix is multiplied by
$a_L^{k-N}$ and therefore tends to zero, while the tail is at most $\epsilon/2$.  Hence
$Y_r(k)\to0$ for every $r$, which proves (f).

Finally, if $\rho\equiv0$, (d) and \eqref{eq:linear_Delta_state_bound} give
\eqref{eq:linear_fulltime_AS_bound}.  The factor $a_L^k$ tends to zero and the multiplier
$\kappa_L+d_\Delta$ is independent of the initial condition.  Standard interpolation of the
integer iterates $a_L^k$ and the two quadratic comparison inequalities
\eqref{eq:linear_V_quadratic_bounds} give a $\KL$ estimate and therefore asymptotic stability.
\end{proof}

\begin{remark}
%[Verification of the state-dependent intermediate offset]
Condition~\eqref{eq:linear_Delta_state_bound} is the linear-network counterpart of the
state-dependent offset condition in Proposition~\ref{prop:practical_form_rigorous}(v).  A directly
checkable sufficient condition is
 $\delta_J^L(0)\leq c_JV(x(0)),
 \varepsilon_L(0)\leq c_\varepsilon\sqrt{V(x(0))}$,
for constants $c_J,c_\varepsilon\geq0$.  Then
 $d_\Delta=(1+\nu_{\rm int})c_J
 +(1+\nu_{\rm int}^{-1})c_Pc_\varepsilon^2$
satisfies \eqref{eq:linear_Delta_state_bound}.  Strong regularity of a local parametric QP can
help establish the second estimate through local Lipschitz dependence of the optimizer and the
predicted trajectory.  It does not, by itself, prove the first estimate because the OCP also
contains the linear penalty on the tightening radii; the cost-excess bound must therefore be
verified separately before a stability claim is made.
\end{remark}

%\begin{remark}
%Conditions \eqref{eq:linear_shift_row_check} and \eqref{eq:linear_completion_socp_full} are evaluated before accepting the shifted candidate at time $t+1$.  Theorem~\ref{thm:cln_RF_stability} proves recursive feasibility on every execution for which both tests pass.  An unconditional implementation theorem additionally requires a specified recursively feasible fallback whenever either test fails.
%\end{remark}

%\input{Case_Study_dcmg}
\section{Case Study: Current Sharing in a Two-DGU DC Microgrid}
\label{sec:case_study_twodgu}

This section specializes Algorithm~\ref{alg:dist_multi_step_mpc} and the constrained-linear analysis of Section~\ref{sec:Constrained linear networks} to the two-DGU prototype.
The DC microgrid is composed of two Distributed Generation Units (DGUs) coupled through resistive ($R_{12}$) power lines to each other.
The energy source of each DGU is provided to the load through a compatible DC-DC power converter TPS55289.
We model the DC load as a constant current source which is located at the terminal bus, so-called Point of Common Coupling (PCC).
In our hardware, a constant current source is realized by one adjustable linear voltage regulator STMicroelectronics LM317 TO-220 (LM317T).
For each DGU, the terminal bus voltage is measured by a A/D voltage sensor ADS1115. 
The DC/DC converter is connected via an RLC filter to the load, the current flowing through the series RL part is measured by a current sensor INA228. 
Fig.~\ref{fig:dcmg_circuit} illustrates an electrical diagram of  DGU $i$ connecting over $R_{ij}$ to the grid.
\begin{figure}[t]
\centering
\includegraphics[width=0.9\linewidth]{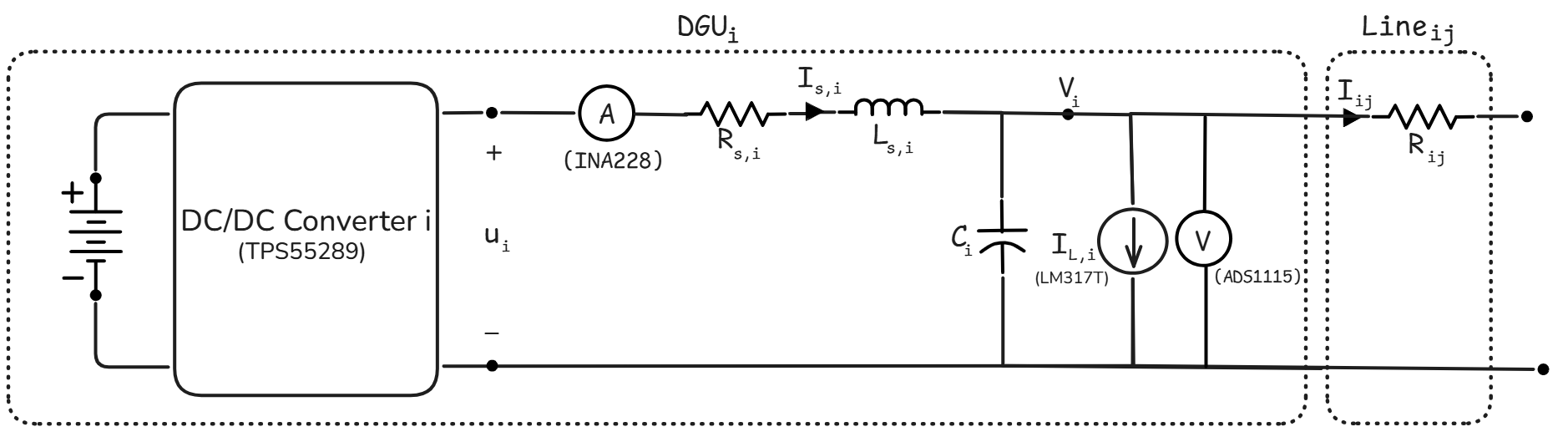}
\caption{Electrical scheme of DGU $i$ linked to line $ij$.}
\label{fig:dcmg_circuit}
\end{figure}
The application has three distinct objectives: source-current sharing, hard bus voltage safety, and hard converter-output (actuator) constraint.

\subsection{Plant model, sharing equilibrium, and constrained control objectives}
\label{subsec:twodgu_objectives}

For DGU~$i$, let $I_{s,i}$ denote the generated source current, $V_i$ the PCC voltage, $u_i$ the controlled converter output voltage, and $I_{L,i}$ the load current.
Under the quasi-stationary-line approximation~\cite{Venkatasubramanian.1995}, we have
\begin{align}
L_{s,i}\dot I_{s,i} &= -R_{s,i}I_{s,i}-V_i+u_i,\label{eq:twodgu_ct_Is_new}\\
C_i\dot V_i &= I_{s,i}-I_{L,i}-\sum_{j\in\mathcal N_i}\frac{V_i-V_j}{R_{ij}}.\label{eq:twodgu_ct_V_new}
\end{align}
For the two-DGU line,
\begin{equation}
\mathcal B=\begin{bmatrix}-1\\1\end{bmatrix},\qquad
\mathcal L:=\mathcal B R_{12}^{-1}\mathcal B^\top
=\frac{1}{R_{12}}
\begin{bmatrix}1&-1\\-1&1\end{bmatrix}.
\end{equation}
With $I_s=(I_{s,1},I_{s,2})^\top$, $V=(V_1,V_2)^\top$, and $u=(u_1,u_2)^\top$, the stacked model is
\begin{equation}\label{eq:twodgu_system_new}
L_s\dot I_s=-R_sI_s-V+u,\qquad
C\dot V=I_s-I_L-\mathcal LV.
\end{equation}
The local physical state is $x_i=(I_{s,i},V_i)^\top$.
For a positive diagonal weighting matrix $W={\rm diag}(w_1,w_2)$, the desired weighted source-current allocation is
\begin{equation}\label{eq:twodgu_current_sharing_objective_new}
\bar I_s
=W^{-1}\mathds 1_2\,
\frac{\mathds 1_2^\top I_L}{\mathds 1_2^\top W^{-1}\mathds 1_2}.
\end{equation}
It is the unique vector satisfying $\mathds 1_2^\top\bar I_s=\mathds 1_2^\top I_L$, $w_1\bar I_{s,1}=w_2\bar I_{s,2}$.
The remaining two objectives are the hard constraints
$V^{\min}\le V_i(t)\le V^{\max}$,
$u^{\min}\le u_i(t)\le u^{\max}$, $i=1,2$.
A steady state compatible with \eqref{eq:twodgu_current_sharing_objective_new} satisfies
\begin{equation}\label{eq:twodgu_equilibrium_conditions_new}
\mathcal LV^\star=\bar I_s-I_L,
\qquad
u^\star=R_s\bar I_s+V^\star.
\end{equation}
Because $\mathds 1_2^\top(\bar I_s-I_L)=0$, the first equation is solvable.
The additive voltage degree of freedom is fixed by the gauge condition $\mathds 1_2^\top V^\star=2V^{\rm nom}$.

\begin{table}[t]
\centering
\caption{Two-DGU prototype parameters used for simulation and hardware validation.}
\label{tab:dgu_params}
\begin{tabular}{c c c c c}
\hline
\textbf{DGU} & $R_{s,i}$ [$\Omega$] & $L_{s,i}$ [mH] & $C_i$ [mF] & $I_{L,i}$ [A]\\
\hline
1 & 0.83 & 3.3 & 8.0 & 0.320\\
2 & 0.83 & 3.3 & 8.0 & 0.356\\
\hline
\end{tabular}

\vspace{0.6em}

\begin{tabular}{c c c c c}
\hline
\textbf{Line} & $R_{12}$ [$\Omega$] & $T$ [s] & $[V^{\min},V^{\max}]$ [V] & $[u^{\min},u^{\max}]$ [V]\\
\hline
(1,2) & 1.20 & 0.100 & [11,13] & [0,15]\\
\hline
\end{tabular}
\end{table}

For the load vector and equal sharing weights $I_L=(0.320,0.356)^\top\,{\rm A},\qquad W=I_2$,
we obtain $\bar I_s=(0.338,0.338)^\top\,{\rm A}$.
With $V^{\rm nom}=12\,{\rm V}$, equations~\eqref{eq:twodgu_equilibrium_conditions_new} give
\begin{equation}\label{eq:twodgu_equilibrium_values_new}
V^\star=(12.01,11.99)^\top\,{\rm V},
\qquad
u^\star=(12.29,12.27)^\top\,{\rm V}.
\end{equation}
Both voltage and input equilibria lie strictly inside the admissible intervals.

The physical load current $I_{L,i}$ remains in the voltage dynamics, whereas the source-current stage and terminal terms are centered at $\bar I_{s,i}$.  Replacing $\bar I_{s,i}$ by $I_{L,i}$ would impose local load tracking and would eliminate the intended power exchange through the tie line.

\begin{remark}\label{rem:twodgu_closed_set}
Strictly speaking, the current sharing together with voltage safety objectives can be formulated as stabilization with respect to the closed set to $\mathcal A_i:=\{(I_{s,i},V_i)\in\mathbb R^2:I_{s,i}=\bar I_{s,i}\}$.
Following the results in~\cite{Noroozi.2018,Noroozi.2020}, one can easily extend the statements in the preceding sections to the closed-set setting by replacing state norms with the so-called measurement functions.
\end{remark}

\subsection{Summary of the finite-step, recursive-feasibility, and stability checks}
\label{subsec:twodgu_certificate_summary}

Let
$\tilde I_s:=I_s-\bar I_s$, $\tilde V:=V-V^\star$, $\tilde u:=u-u^\star$,
and define the interleaved shifted state
$\tilde x=(\tilde I_{s,1},\tilde V_1,\tilde I_{s,2},\tilde V_2)^\top$.
Exact zero-order-hold discretization of \eqref{eq:twodgu_system_new} at $T=0.1\,\mathrm{s}$ gives
$\tilde x(t+1)=A\tilde x(t)+B\tilde u(t)$.  With $e_i:=\tilde I_{s,i}$ we use
$W_i=e_i^2$, $i=1,2$, which are positive definite with respect to the current-sharing sets in
Remark~\ref{rem:twodgu_closed_set}.  The argument below is therefore the closed-set,
current-error specialization of the quadratic analysis in
Section~\ref{sec:Constrained linear networks}.  
Due to space constraints, we don't present the offline feasibility and stability checks.
Full-precision matrices, modal-coordinate
calculations, invariant-set computations, and the numerical interval checks are available through a user-friendly Python-native simulation environment:
\href{https://github.com/navidnoroozi/dmpc-4-dcmg}{\texttt{dmpc-4-dcmg}} as the paper's appendix. 
For more details, check out the git repo instruction \href{https://github.com/navidnoroozi/dmpc-4-dcmg/blob/main/Python_env/HOW_TO_RUN_WORKFLOW.md}{\texttt{README}} supported with a simulation workflow (copy/paste) guideline \href{https://github.com/navidnoroozi/dmpc-4-dcmg/blob/main/Python_env/HOW_TO_RUN_WORKFLOW.md}{\texttt{HOW-TO-RUN-WORKFLOW}}.

A modal-coordinate invariance calculation gives a compact domain on which
$11\leq V_i\leq13$ and
$11.3197\leq u_i\leq13.2414\subset[0,15]$.
The recursive-feasibility tests are scalar.
With
$\delta_i(l|t+1):=\chi_i(l|t+1)+\tilde r_i(l|t+1)$, the shift-margin condition is exactly
\begin{equation}\label{eq:twodgu_exact_shift_check_new}
 V^{\min}+\delta_i(l|t+1)
 \leq V_i^{\nom,\star}(l+1|t)
 \leq V^{\max}-\delta_i(l|t+1),
 \qquad l=1,\ldots,M-1.
\end{equation}
For the final input of the shifted candidate, the predicted current error has the affine form
$e_i^+=\zeta_i(t+1)+\beta_i v_i$, with $\beta_i=0.35\neq0$.  Writing
$h_i:=\sqrt{b_i(t+1)}$, the one-step terminal inequality is feasible if and only if
\begin{equation}\label{eq:twodgu_completion_interval_new}
[u^{\min},u^{\max}]
\cap
\left[
\min\!\left\{\frac{-h_i-\zeta_i}{\beta_i},
              \frac{ h_i-\zeta_i}{\beta_i}\right\},
\max\!\left\{\frac{-h_i-\zeta_i}{\beta_i},
              \frac{ h_i-\zeta_i}{\beta_i}\right\}
\right]\neq\varnothing.
\end{equation}
Thus \eqref{eq:twodgu_exact_shift_check_new} and
\eqref{eq:twodgu_completion_interval_new} are exact local tests of the two conditions used in
Theorem~\ref{thm:online_tube_rec_feas}.

We next specialize the corrected stability bounds.  Define
%\begin{equation}\label{eq:twodgu_rho_new} 
$\rho_I(t):=\max_{i=1,2}\bigl(r_i^\star(2|t)+s_i(2|t)\bigr)$.
%\end{equation}
Because the current component of the local state error is no larger than its Euclidean norm,
Lemma~\ref{lem:corrected_actual_nominal_bound} implies
$|e_i(t+2)-e_i^{\nom,\star}(2|t)|\leq\rho_I(t)$.  Therefore, for every $\nu_I>0$,
\begin{equation}\label{eq:twodgu_practical_current_estimate_new}
 V_I(e(t+2))
 \leq (1+\nu_I)\theta V_I(e(t))
 +(1+\nu_I^{-1})\rho_I(t)^2.
\end{equation}
Choose $\nu_I$ such that
$a_I:=(1+\nu_I)\theta<1$ and set $b_I:=1+\nu_I^{-1}$.
Unlike former compact-set Lipschitz estimate, \eqref{eq:twodgu_practical_current_estimate_new}
is a direct quadratic bound and doesn't require a term proportional to
$\bar e\,\rho_I$.
The intermediate physical time is treated in the same way as Lemma~\ref{lem:linear_intermediate_bound}.
Let $E_i^I$ select the $i$th current error from $e$ and define
%\begin{equation}\label{eq:twodgu_cq}
$c_{I,i}:=\sum_{l=0}^{1}\|E_i^IF^l\|_2^2$, $c_I:=\max_i c_{I,i}$.
%\end{equation}
For the local optimal values $J_i^\star(t)$ define
%\begin{align}\label{eq:twodgu_deltaJ}
$\delta_{J,I}(t):=\max_{i=1,2}\max\{J_i^\star(t)-c_{I,i}\|e(t)\|_2^2,0\}$,
%\end{align}
and
%\begin{equation}\label{eq:twodgu_epsilon_int}
 $\varepsilon_I(t):=
 \max_{r=0,1}\max_{i=1,2}
 \bigl(r_i^\star(r|t)+s_i(r|t)\bigr)$.
%\end{equation}
For any fixed $\nu_{I,\rm int}>0$, put
\begin{equation}\label{eq:twodgu_kappa_Delta}
 \kappa_I:=2(1+\nu_{I,\rm int})c_I,
 \qquad
 \Delta_I(t):=(1+\nu_{I,\rm int})\delta_{J,I}(t)
 +(1+\nu_{I,\rm int}^{-1})\varepsilon_I(t)^2.
\end{equation}
Since $\|e\|_2^2\leq2V_I(e)$, the same cost-comparison and quadratic perturbation argument as
in Lemma~\ref{lem:linear_intermediate_bound} gives
\begin{equation}\label{eq:twodgu_intermediate_bound}
 V_I(e(t+r))\leq\kappa_I V_I(e(t))+\Delta_I(t),
 \qquad r=0,1.
\end{equation}

\begin{proposition}
\label{prop:twodgu_rf_stability_new}
Consider the certified $M=2$ design above.  Suppose the two local OCPs are feasible initially,
\eqref{eq:twodgu_exact_shift_check_new} and \eqref{eq:twodgu_completion_interval_new} hold at
every transition, the closed-loop states remain in the operating domain on which the finite-step
witness is admissible, and the dynamic terminal scalars of
Assumption~\ref{ass:terminal_comm} are used.  Then:
\begin{enumerate}
\item[(i)] the OCPs are recursively feasible and, for all $t$,
$11\leq V_i(t)\leq13$ and $0\leq u_i(t)\leq15$, $i=1,2$;
\item[(ii)] the current errors satisfy the intermediate-time estimate
\eqref{eq:twodgu_intermediate_bound} and the two-step estimate
\begin{equation}\label{eq:twodgu_a_b_recurrence}
 V_I(e(t+2))\leq a_I V_I(e(t))+b_I\rho_I(t)^2;
\end{equation}
\item[(iii)] if $\rho_I(t)\leq\bar\rho_I$ for all $t$, then
%\begin{equation}\label{eq:twodgu_practical_limsup_new}
 $\limsup_{t\to\infty}V_I(e(t))
 \leq\frac{b_I}{1-a_I}\bar\rho_I^2,
 \limsup_{t\to\infty}\|e(t)\|_\infty
 \leq\sqrt{\frac{b_I}{1-a_I}}\,\bar\rho_I$;
%\end{equation}
\item[(iv)] if $\rho_I(t)\to0$, then $I_s(t)\to\bar I_s$;
\item[(v)] if $\rho_I(t)=0$ for all $t$ and there exists $d_{\Delta,I}\geq0$ such that
$\Delta_I(0)\leq d_{\Delta,I}V_I(e(0))$, then the current-sharing set is asymptotically stable
with respect to the measurement $\|e\|_\infty$.  In particular,
\begin{equation}\label{eq:twodgu_fulltime_AS}
 V_I(e(2k+r))
 \leq a_I^k(\kappa_I+d_{\Delta,I})V_I(e(0)),
 \qquad r\in\{0,1\}.
\end{equation}
\end{enumerate}
No convergence of the bus-voltage vector to the particular gauge-selected $V^\star$ is asserted.
\end{proposition}

%\begin{proof}
%Item (i) is Theorem~\ref{thm:online_tube_rec_feas} specialized through the scalar tests \eqref{eq:twodgu_exact_shift_check_new}--\eqref{eq:twodgu_completion_interval_new}, followed by Corollary~\ref{cor:true_constraint_satisfaction}.  Equation \eqref{eq:twodgu_intermediate_bound} was established above and \eqref{eq:twodgu_a_b_recurrence} is \eqref{eq:twodgu_practical_current_estimate_new} with the definitions of $a_I,b_I$.  Iteration on the two interlaced sequences $\{2k\}$ and $\{2k+1\}$ gives the geometric-series bound in (iii).  If $\rho_I(t)\to0$, the corresponding stable convolution tends to zero, proving (iv).  Under the hypotheses of (v), the mismatch term in the two-step recurrence is zero and \eqref{eq:twodgu_intermediate_bound} at $t=0$ gives the initial bound for both interlaced sequences; hence \eqref{eq:twodgu_fulltime_AS} follows.  This is a $\KL$ estimate for $\|e\|_\infty$ after the standard interpolation of $a_I^k$.
%\end{proof}

\subsection{Simulation and Hardware Validation Results}
\label{subsec:twodgu_simulation_hardware_results}

Our prototype enjoys a modular, layered, traceable software architecture.
In particular, the prototype uses a hybrid communication network topology consisting of one Raspberry Pi~4 (RPi4) hosting local MPC controller per DGU with P2P communication between the controllers, and an Ubuntu Virtual Machine (VM) hosting the ZeroMQ proxy, coordinator and logger nodes.
The coordinator also hosts the load-current high-gain observer.
Each RPi4 publishes timestamped state, prediction, heartbeat, and diagnostic messages to the proxy's XSUB ingress on port \(6000\), and subscribes through its XPUB egress on port \(6001\) to the peer prediction and coordinator messages.
On each DGU-side, the RPi4 exchanges binary command/state packets with an ESP32 microcontroller over \(921600\)-baud USB serial.
The ESP32 owns the I\(^2\)C links to the TPS55289 actuator, INA228 current sensor, and ADS1115 bus-voltage sensor, while enforcing timestamp, watchdog, and safety checks.
Our proposed Algorithm~\ref{alg:dist_multi_step_mpc}, running on RPi4, is numerically solved by the CVXPY package in Python.
Electrically, the two bus nodes are coupled through the \(1.20\,\Omega\) line resistor.
The two-DGU DC microgrid prototype is shown in Fig.~\ref{fig:twodgu_hardware_architecture}.
Fig.~\ref{fig:twodgu_software_architecture} illustrates this modular, layered, and traceable implementation.

\begin{figure}[!t]
\centering
\includegraphics[width=0.8\linewidth]{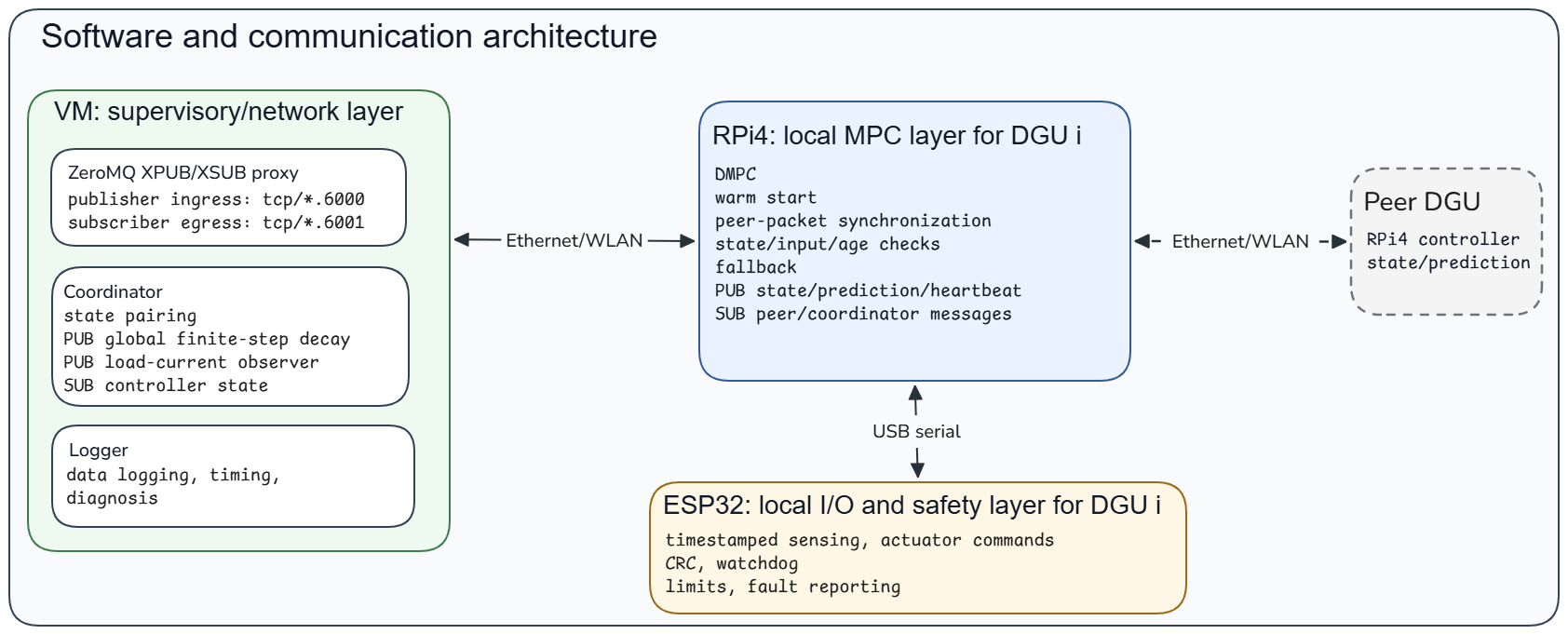}
\caption{Software and communication architecture around one representative DGU.  The peer MPC controller, coordinator, and logger communicate through the VM-side ZeroMQ XPUB/XSUB proxy; the RPi4 delegates deterministic I/O and local safety to the ESP32.}
\label{fig:twodgu_software_architecture}
\end{figure}

\begin{figure}[!t]
\centering
\includegraphics[height=0.5\linewidth]{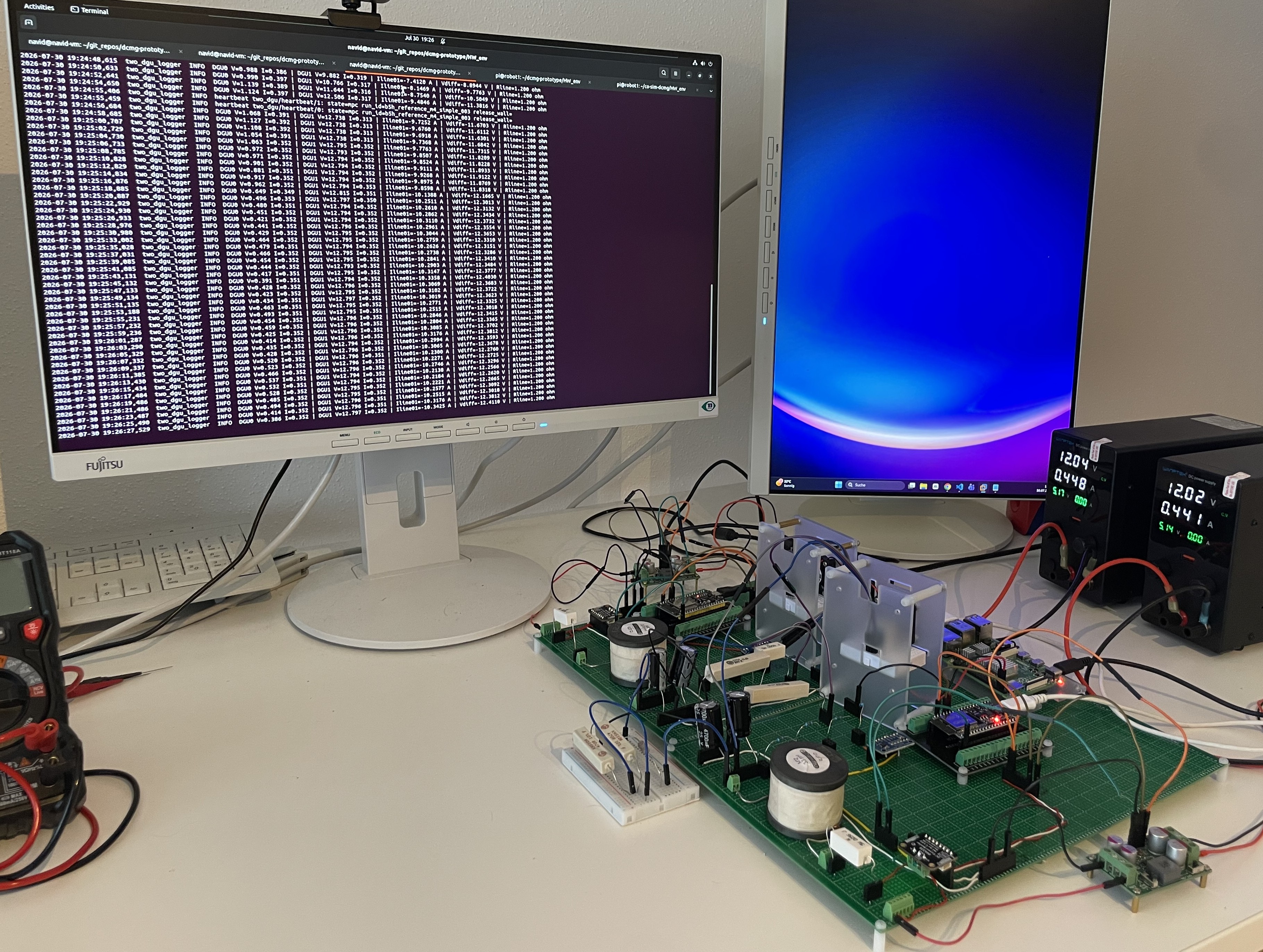}
\caption{Two-DGU DC microgrid laboratory prototype.}
\label{fig:twodgu_hardware_architecture}
\end{figure}

This subsection compares the simulation with the two-DGU laboratory prototype under the same nominal electrical parameters in Table~\ref{tab:dgu_params}, equal current-sharing weights, a $100\,\mathrm{ms}$ sampling period, and a practical prediction horizon of $M=4$.
While Proposition~\ref{prop:twodgu_rf_stability_new} holds $M=2$, one can see that the same feasibility and stability results, probably with larger bounds, also satisfied for any multiple of $n M$, $n \in \N$.
We take $M=4$ to enjoy advantages of longer horizon~\cite{Grune.2017a}.
A $235\,\Omega$ resistor was connected at DGU~1 during the shaded interval in Figs.~\ref{fig:twodgu_source_currents}--\ref{fig:twodgu_load_observer_error}.
At a $12\,\mathrm{V}$ bus voltage this perturbation represents approximately $51\,\mathrm{mA}$ of additional load current.  The simulation and hardware traces are shifted to the common comparison coordinate $\tau=0$ at the selected load-connection instant; the plotted pulse window is approximately $28.26\,\mathrm{s}$.

Fig.~\ref{fig:twodgu_source_currents} shows the source-current response before, during, and after the load pulse.  In the nominal exact-model simulation, both source currents are initially at $0.338\,\mathrm{A}$.  After the shunt is connected, the load observer raises the equal-sharing reference to approximately $0.363\,\mathrm{A}$, and the two simulated currents settle at the same value over the interior of the pulse interval.  Their mean tracking errors in this steady pulse window are below $1\,\mu\mathrm{A}$ in magnitude, and the root-mean-square weighted sharing mismatch $I_{s,1}-I_{s,2}$ is approximately $0.198\,\mathrm{mA}$.  The short oscillations at pulse connection and removal are bounded and the simulated currents remain between $0.327\,\mathrm{A}$ and $0.374\,\mathrm{A}$.

The hardware trajectories exhibit the same qualitative response. Before the disturbance, the mean measured currents are both approximately $0.335\,\mathrm{A}$.
In the steady part of the pulse they rise to $0.3553\,\mathrm{A}$ and $0.3554\,\mathrm{A}$, respectively.  The corresponding observer-based current reference is approximately $0.3588\,\mathrm{A}$, leaving mean absolute tracking offsets of about $3.48\,\mathrm{mA}$ and $3.43\,\mathrm{mA}$.  Despite this common-mode offset, the sharing objective is achieved accurately: the steady-pulse root-mean-square value of $I_{s,1}-I_{s,2}$ is approximately $0.571\,\mathrm{mA}$. Thus the principal distributed objective---equal source-current allocation for $W=I_2$---is substantially more accurate than the absolute current-reference tracking on the physical prototype.  The difference is consistent with current-sensor uncertainty, load-estimation bias, converter nonlinearities, and plant-parameter mismatch that are absent from the exact-model simulation.

\begin{figure}[!t]
\centering
\includegraphics[height=0.4\linewidth]{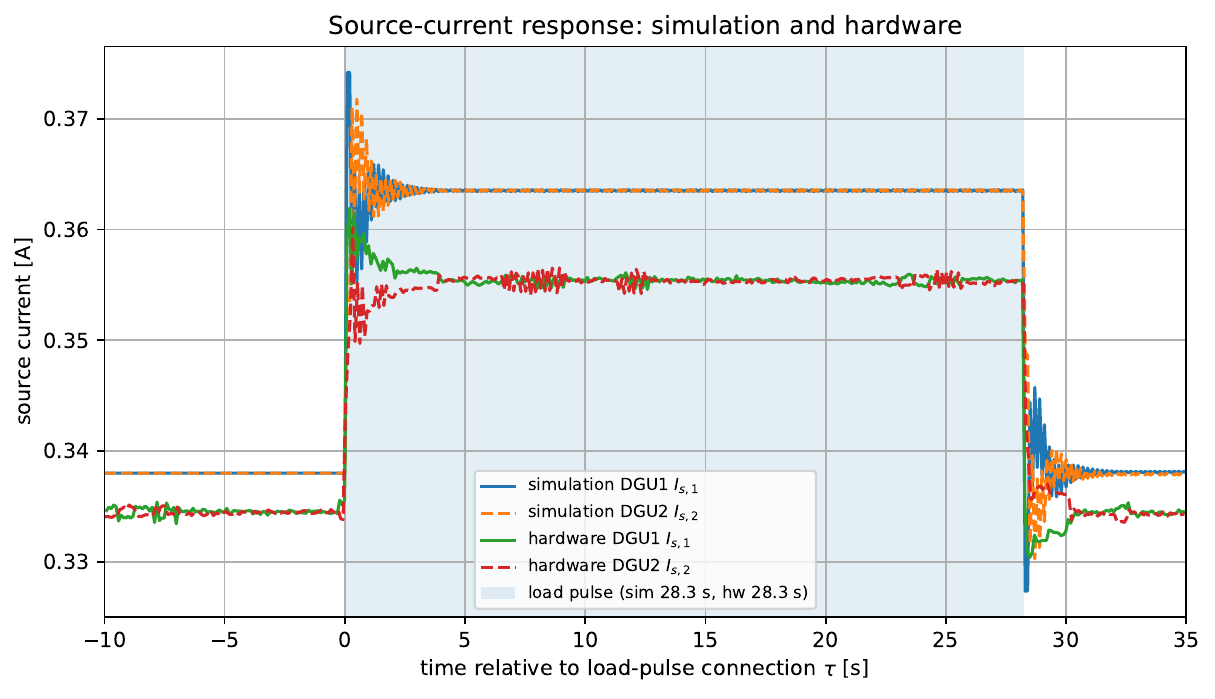}
\caption{Source-current response in simulation and hardware, aligned at the connection of the $235\,\Omega$ load at DGU~1.  Both implementations redistribute the additional load and recover the nominal sharing level after removal.}
\label{fig:twodgu_source_currents}
\end{figure}

%\subsubsection{Voltage safety and converter commands}
The voltage trajectories in Fig.~\ref{fig:twodgu_voltages} remain strictly inside the prescribed interval $[11,13]\,\mathrm{V}$.  In simulation, the complete-run voltage ranges are $[11.981,12.026]\,\mathrm{V}$ for DGU~1 and $[11.98,12.02]\,\mathrm{V}$ for DGU~2.  During the steady portion of the pulse their means are approximately $11.995\,\mathrm{V}$ and $12.004\,\mathrm{V}$, respectively.  The exact-model simulation therefore remains very close to the gauge-selected operating voltage.
The prototype has larger steady offsets, but preserves the safety objective.  Over the complete hardware MPC phase, the measured ranges are $[11.458,11.654]\,\mathrm{V}$ for DGU~1 and $[11.752,11.863]\,\mathrm{V}$ for DGU~2.  Their steady-pulse means are approximately $11.50\,\mathrm{V}$ and $11.79\,\mathrm{V}$.
The experiment therefore confirms constrained voltage regulation and safety, not convergence to the particular voltage gauge in \eqref{eq:twodgu_equilibrium_values_new}.
\begin{figure}[!t]
\centering
\includegraphics[height=0.4\linewidth]{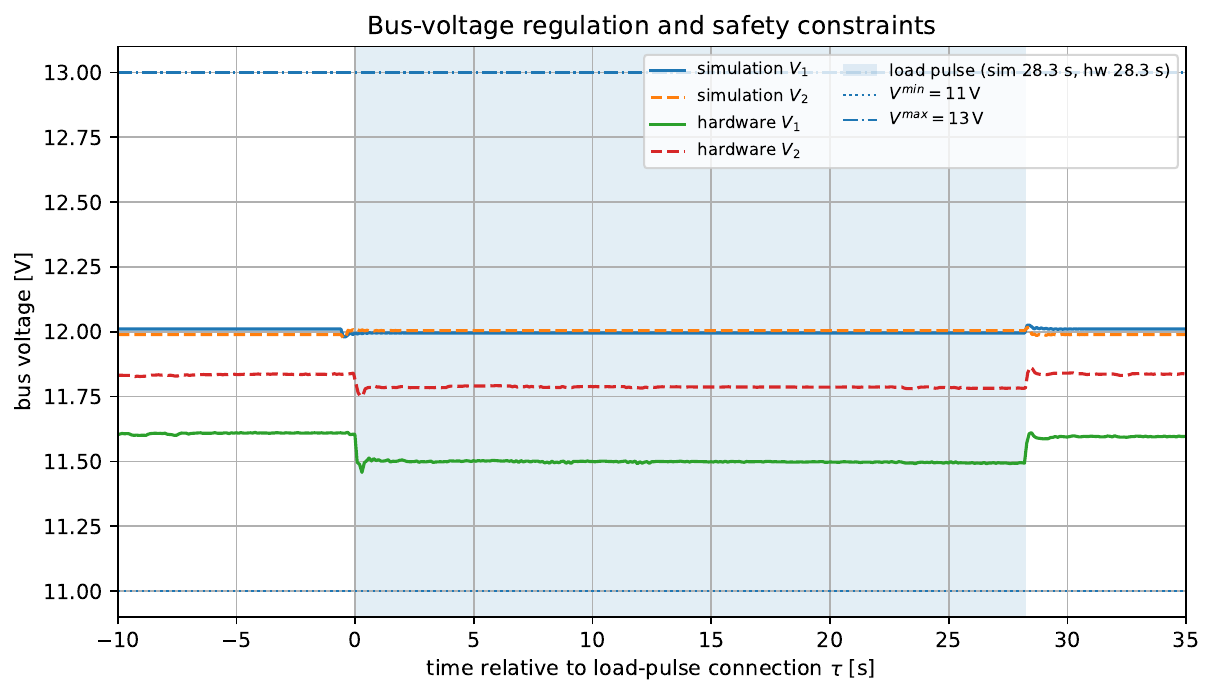}
\caption{Bus-voltage trajectories and the hard limits $V^{\min}=11\,\mathrm{V}$ and $V^{\max}=13\,\mathrm{V}$.  The simulation remains close to the nominal operating voltage, whereas the prototype exhibits larger offsets while preserving the safety interval.}
\label{fig:twodgu_voltages}
\end{figure}
Fig.~\ref{fig:twodgu_inputs} shows that the required control action also remains well inside the converter limits.  The simulated DGU commands range over $[12.28,12.31]\,\mathrm{V}$ and $[12.26,12.32]\,\mathrm{V}$.  The corresponding hardware ranges are $[12.09,12.29]\,\mathrm{V}$ and $[12.27,12.50]\,\mathrm{V}$.  Hence all commands retain margins of more than $2.5\,\mathrm{V}$ from the upper limit and more than $12\,\mathrm{V}$ from the lower limit.  The small bounded changes at the two switching instants demonstrate that the current redistribution is achieved without aggressive or saturating converter commands.

\begin{figure}[!t]
\centering
\includegraphics[height=0.4\linewidth]{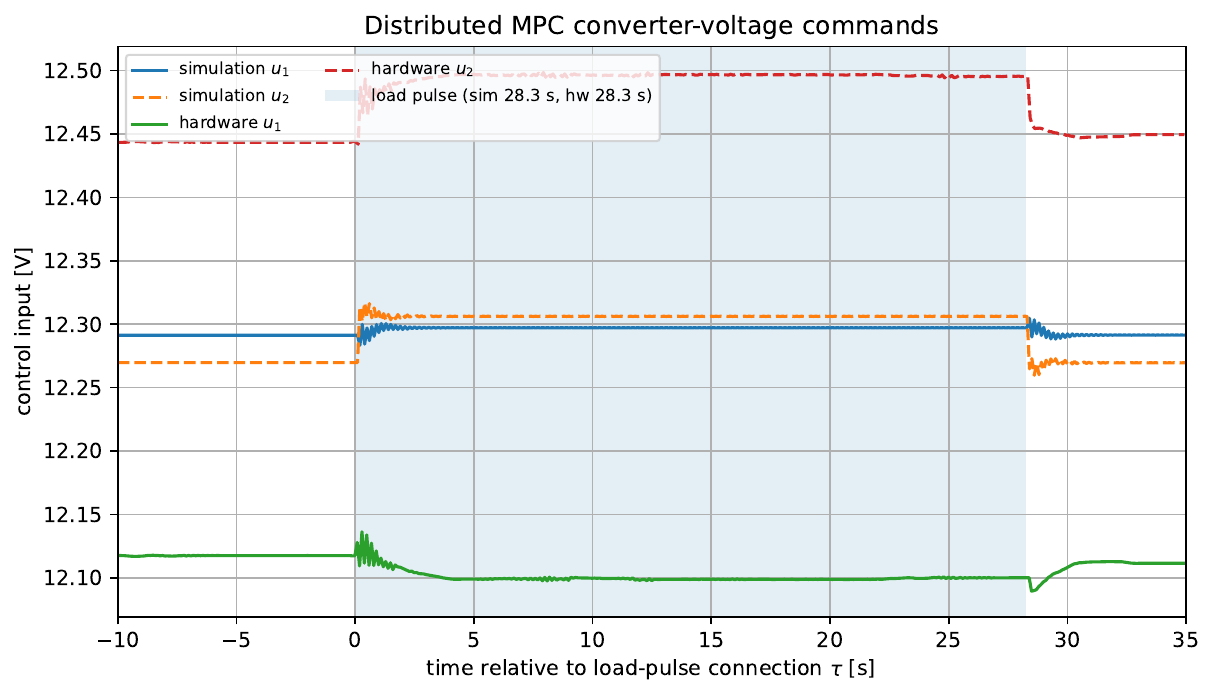}
\caption{DMPC converter-voltage commands.  Both the exact-model simulation and the hardware prototype remain strictly inside the admissible interval $[0,15]\,\mathrm{V}$.}
\label{fig:twodgu_inputs}
\end{figure}

Fig.~\ref{fig:twodgu_load_observer_error} compares the estimated load current with the known simulated load and with the hardware load profile reconstructed from the nominal current sink plus the measured bus voltage divided by $235\,\Omega$.

\begin{figure}[!t]
\centering
\includegraphics[height=0.4\linewidth]{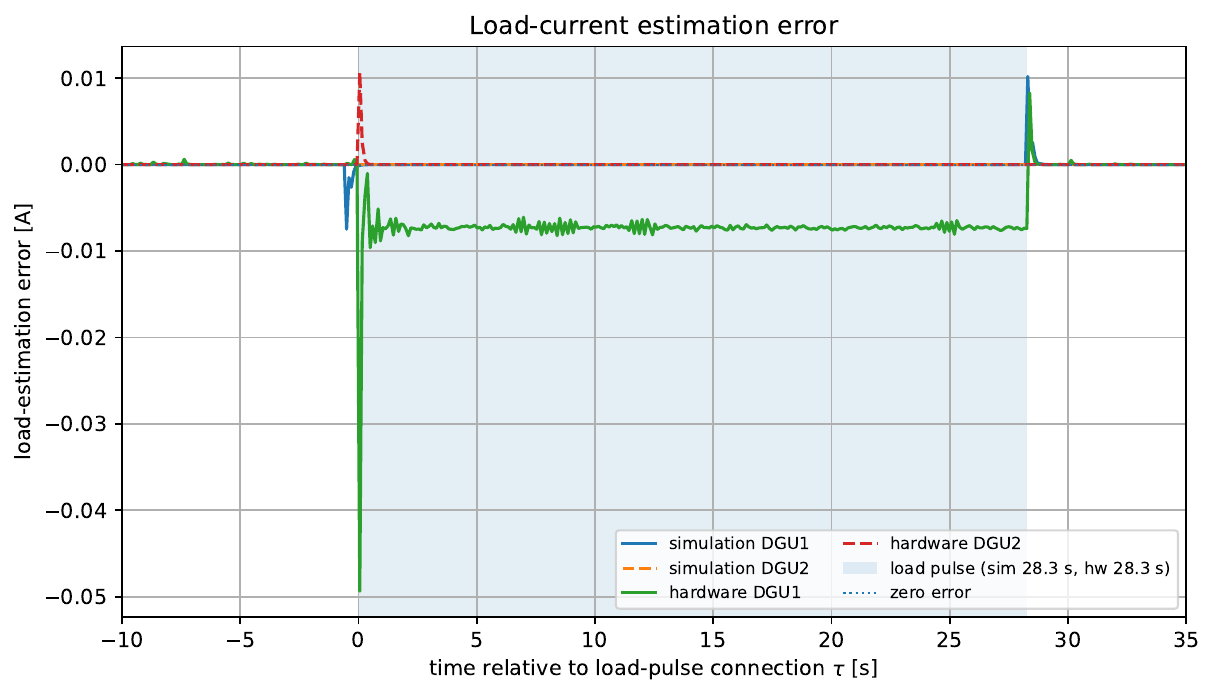}
\caption{Load-current estimation errors.  The exact-model simulation converges essentially to zero error; the prototype exhibits a bounded DGU~1 underestimation during the shunt pulse and negligible steady error at the undisturbed DGU~2.}
\label{fig:twodgu_load_observer_error}
\end{figure}

%\subsubsection{Timing, solver health, and synchronization integrity}
Table~\ref{tab:twodgu_validation_timing} summarizes the real-time and solver telemetry.  The strict simulation run contains exactly $1200$ samples for each controller, the plant, the coordinator observer, and the balance logger.  Every sequence covers steps $0$ through $1999$ without a duplicate or missing step.  The plant receives both commands at every step, the coordinator forms only exact same step state pairs with zero simulated-time skew, and no controller uses a synchronization fallback.  All $2400$ local QPs report an optimal status. 
Both controllers report zero solver failures and zero deadline misses.  The $95$th-percentile solve times are $20.045\,\mathrm{ms}$ and $20.817\,\mathrm{ms}$, while the maximum complete loop execution times are $24.633\,\mathrm{ms}$ and $25.446\,\mathrm{ms}$. 

\begin{table}[!t]
\caption{Timing and solver summary for the $M=4$, $T=100\,\mathrm{ms}$ validation runs.}
\label{tab:twodgu_validation_timing}
\centering
\small
\begin{tabular}{llrrrrrr}
\hline
Platform & DGU & Samples & Solve p95 & Loop p95 \\
 & & & [ms] & [ms] \\
\hline
Simulation & 1 & 1200 & 4.087 & -- \\
Simulation & 2 & 1200 & 4.376 & -- \\
Hardware & 1 & 1200 & 20.045 &  24.633 \\
Hardware & 2 & 1200 & 20.817 & 25.446 \\
\hline
\end{tabular}
\end{table}

\section{Conclusions}
This paper extended the constructive finite-step Lyapunov framework from centralized multi-step MPC to DMPC for constrained, dynamically coupled discrete-time systems.
The central motivation was that converse finite-step Lyapunov and non-conservative small-gain results permit simple local candidates and compatible gains to be constructed for a sufficiently large finite step $M$, without requiring one-step stabilizability of every subsystem.  The distributed difficulty is that each controller uses exchanged neighbor predictions.
The local OCPs use optimized state-constraint tightening radii and a finite-step small-gain terminal condition.  Recursive feasibility follows from shift-compatible state margins and a one-step terminal feasibility test.  The stability analysis separately bounds the prediction and reoptimization mismatch.  It yields an $M$-step practical Lyapunov estimate and an intermediate-time estimate with the explicit offset; a bounded terminal discrepancy gives practical boundedness, a vanishing discrepancy gives convergence.
The asymptotic stabilization is obtained when the intermediate-time offset vanishes with the initial Lyapunov level.
For constrained linear networks, the general bounds reduce to convex program tests, with a direct quadratic estimate in the terminal discrepancy.
The two-DGU specialization verifies the corresponding current-sharing and recursive-feasibility conditions.
Moreover, simulation and hardware experiments demonstrate current sharing, constraint satisfaction, feasible local optimization, and sufficient real-time margin at the sampling period.

\bibliographystyle{IEEEtran}
\bibliography{DistrubtedMutiStepMPC}

@ARTICLE{Nasirian.2014,
  author={Nasirian, Vahidreza and Moayedi, Seyedali and Davoudi, Ali and Lewis, Frank L.},
  journal={IEEE Transactions on Power Electronics}, 
  title={Distributed Cooperative Control of DC Microgrids}, 
  year={2015},
  volume={30},
  number={4},
  pages={2288-2303},
  doi={10.1109/TPEL.2014.2324579}}

@inproceedings{Noroozi2018IFACDCMicrogridMPC,
  title     = {Model predictive control of {DC} microgrids: current sharing and voltage regulation},
  author    = {Noroozi, Navid and Trip, Sebastian and Geiselhart, Roman},
  booktitle = {7th {IFAC} Workshop on Distributed Estimation and Control in Networked Systems},
  journal   = {{IFAC-PapersOnLine}},
  volume    = {51},
  number    = {23},
  pages     = {124--129},
  year      = {2018},
  doi       = {10.1016/j.ifacol.2018.12.022}
}

@article{Venkatasubramanian.1995,
	Author = {V. Venkatasubramanian and H. Schattler and J. Zaborszky},
	Journal = {IEEE Trans. Autom. Control},
	Number = {11},
	Pages = {1975--1982},
	Title = {Fast time-varying phasor analysis in the balanced three-phase large electric power system},
	Volume = {40},
	Year = {1995}}

@article{Sandberg2015CyberphysicalSurvey,
  title={Cyberphysical security in networked control systems: An introduction to the issue},
  author={Sandberg, Henrik and Amin, Saurabh and Johansson, Karl Henrik},
  journal={IEEE Control Systems Magazine},
  volume={35},
  number={1},
  pages={20--23},
  year={2015},
  publisher={IEEE},
  doi={10.1109/MCS.2014.2364708}
}

@Article{Geiselhart.2014c,
  Title                    = {An alternative converse {L}yapunov theorem for discrete-time systems},
  Author                   = {Geiselhart, Roman and Gielen, Rob H. and Lazar, Mircea and Wirth, Fabian R.},
  Journal                  = {Syst. Control Lett.},
  Year                     = {2014},
  Pages                    = {49--59},
  Volume                   = {70}
}

@article{Lucia.2015,
title = {Contract-based Predictive Control of Distributed Systems with Plug and Play Capabilities},
journal = {IFAC-PapersOnLine},
volume = {48},
number = {23},
pages = {205--211},
year = {2015},
note = {5th IFAC Conference on Nonlinear Model Predictive Control NMPC 2015},
author = {Sergio Lucia and Markus K{\"o}gel and Rolf Findeisen},
doi = {10.1016/j.ifacol.2015.11.284}
}

@article{Christofides.2013,
title = {Distributed Model Predictive Control: A Tutorial Review and Future Research Directions},
journal = {Computers \& Chemical Engineering},
volume = {51},
pages = {21--41},
year = {2013},
author = {Panagiotis D. Christofides and Riccardo Scattolini and David {Mu{\~n}oz de la Pe{\~n}a} and Jinfeng Liu},
doi = {10.1016/j.compchemeng.2012.05.011}
}

@Article{Geiselhart.2015,
  Title                    = {A Relaxed Small-Gain Theorem for Interconnected Discrete-Time Systems},
  Author                   = {Geiselhart, Roman and Lazar, Mircea and Wirth, Fabian R.},
  Journal                  = {IEEE Trans. Autom. Control},
  Year                     = {2015},
  Number                   = {3},
  Pages                    = {812--817},
  Volume                   = {60}
}

@Article{Geiselhart.2017,
  Title                    = {Equivalent types of {ISS} {L}yapunov functions for discontinuous discrete-time systems},
  Author                   = {Roman Geiselhart and Navid Noroozi},
  Journal                  = {Automatica},
  Year                     = {2017},
  Pages                    = {227--231},
  Volume                   = {84}
}

@Article{Gielen.2015,
  Title                    = {On stability analysis methods for large-scale discrete-time systems},
  Author                   = {Gielen, Rob H. and Lazar, Mircea},
  Journal                  = {Automatica},
  Year                     = {2015},
  Pages                    = {66--72},
  Volume                   = {55}
}

@Book{Grune.2017a,
  Title                    = {Nonlinear Model Predictive Control: Theory and Algorithms},
  Author                   = {Lars Gr{\"u}ne and J{\"u}rgen Pannek},
  Publisher                = {Springer},
  Year                     = {2017},

  Address                  = {London},
  Edition                  = {2nd}
}

@InProceedings{Hermans.2010,
  Title                    = {{Almost Decentralized Lyapunov-based Nonlinear Model Predictive Control}},
  Author                   = {Hermans, R. M. and Lazar, M. and Joki{\'c}, A.},
  Booktitle                = {2010 American Control Conference},
  Year                     = {2010},
  Pages                    = {3932--3938},
  Doi                      = {10.1109/ACC.2010.5530649}
}

@article{Scattolini.2009,
  author  = {Scattolini, R.},
  title   = {Architectures for Distributed and Hierarchical Model Predictive Control -- A Review},
  journal = {Journal of Process Control},
  year    = {2009},
  volume  = {19},
  number  = {5},
  pages   = {723--731},
  doi     = {10.1016/j.jprocont.2009.02.003}
}

@article{Noroozi.2020,
title = {Control of discrete-time nonlinear systems via finite-step control {L}yapunov functions},
journal = {Systems \& Control Letters},
volume = {138},
pages = {104631},
year = {2020},
issn = {0167--6911},
author = {Navid Noroozi and Roman Geiselhart and Lars Grüne and Fabian R. Wirth}
}

@ARTICLE{Noroozi.2018,
  author={Noroozi, Navid and Geiselhart, Roman and Grüne, Lars and Rüffer, Björn S. and Wirth, Fabian R.},
  journal={IEEE Transactions on Automatic Control}, 
  title={Nonconservative Discrete-Time ISS Small-Gain Conditions for Closed Sets}, 
  year={2018},
  volume={63},
  number={5},
  pages={1231--1242}}

@article{Ruffer.2010,
title = {Monotone inequalities, dynamical systems, and paths in the positive orthant of {E}uclidean n-space},
author                   = {R{\"u}ffer, Bj{\"o}rn S.},
journal                  = {Positivity},
year                     = {2010},
number                   = {2},
pages                    = {257--283},
volume                   = {14}
}

@article{Dunbar.2007,
  author  = {Dunbar, William B.},
  title   = {Distributed Receding Horizon Control of Dynamically Coupled Nonlinear Systems},
  journal = {IEEE Transactions on Automatic Control},
  year    = {2007},
  volume  = {52},
  number  = {7},
  pages   = {1249--1263},
  doi     = {10.1109/TAC.2007.900828}
}

@article{FarinaScattolini.2012,
  author  = {Farina, Marcello and Scattolini, Riccardo},
  title   = {Distributed Predictive Control: A Non-Cooperative Algorithm with Neighbor-to-Neighbor Communication for Linear Systems},
  journal = {Automatica},
  year    = {2012},
  volume  = {48},
  number  = {6},
  pages   = {1088--1096},
  doi     = {10.1016/j.automatica.2012.03.020}
}

@article{MagniScattolini.2006,
  author  = {Magni, Lalo and Scattolini, Riccardo},
  title   = {Stabilizing Decentralized Model Predictive Control of Nonlinear Systems},
  journal = {Automatica},
  year    = {2006},
  volume  = {42},
  number  = {7},
  pages   = {1231--1236},
  doi     = {10.1016/j.automatica.2006.02.010}
}

@article{RaimondoISS.2007,
  author  = {Raimondo, Davide M. and Magni, Lalo and Scattolini, Riccardo},
  title   = {Decentralized {MPC} of Nonlinear Systems: An Input-to-State Stability Approach},
  journal = {International Journal of Robust and Nonlinear Control},
  year    = {2007},
  volume  = {17},
  number  = {17},
  pages   = {1651--1667},
  doi     = {10.1002/rnc.1214}
}

@inproceedings{RaimondoIterative.2009,
  author    = {Raimondo, Davide M. and Hokayem, Peter and Lygeros, John and Morari, Manfred},
  title     = {An Iterative Decentralized {MPC} Algorithm for Large-Scale Nonlinear Systems},
  booktitle = {IFAC Proceedings Volumes},
  year      = {2009},
  volume    = {42},
  number    = {20},
  pages     = {162--167},
  doi       = {10.3182/20090924-3-IT-4005.00028}
}

@article{ZhengSmallGain.2025,
  author  = {Zheng, Yi and Liu, Qibo and Zhao, Qingchun and Wang, Yanye and Li, Shaoyuan},
  title   = {Small-Gain Based Distributed Model Predictive Control of Nonlinear Continuous Processes},
  journal = {Chemical Engineering Research and Design},
  year    = {2025},
  volume  = {223},
  pages   = {177--184},
  doi     = {10.1016/j.cherd.2025.09.026}
}

@article{TripExperimental.2018,
  author  = {Trip, Sebastian and Han, Renke and Cucuzzella, Michele and Cheng, Xiaodong and Scherpen, Jacquelien M. A. and Guerrero, Josep M.},
  title   = {Distributed Averaging Control for Voltage Regulation and Current Sharing in {DC} Microgrids: Modelling and Experimental Validation},
  journal = {IFAC-PapersOnLine},
  year    = {2018},
  volume  = {51},
  number  = {23},
  pages   = {242--247},
  doi     = {10.1016/j.ifacol.2018.12.042}
}

\end{document}